\documentclass[10pt,twoside]{article}
\usepackage{relsize}
\usepackage{artmods,macros}
\usepackage{amsmath,amssymb}
\usepackage[toc,page]{appendix}
\usepackage{algorithm}
\usepackage{algpseudocode}

\title{On Modified Newton Methods for Unconstrained Optimization Based on
         Lanczos Tridiagonalization\thanks{\footWASP}}
\author{Aban ANSARI-\"ONNESTAM\thanks{\footAO} \and Anders
  FORSGREN\thanks{\footAF} \and Philip E.~GILL\thanks{\footPEG}}

\def\footAF{Division of Numerical Analysis, Optimization and Systems
  Theory, Department of Mathematics, KTH Royal Institute of Technology, SE-100 44
            Stockholm, Sweden ({\tt andersf@kth.se}).}

 \def\footAO{Division of Applied Mathematics, Department of
            Mathematics, Link\"oping University, SE-581 83
            Link\"{o}ping, Sweden ({\tt aban.ansari-onnestam@liu.se}). }
\def\footWASP{The work of the first two authors was partially supported by the Wallenberg AI, Autonomous Systems and
Software Program (WASP) funded by the Knut and Alice Wallenberg Foundation.}

\def\footPEG{Department of Mathematics, University of California San Diego,
La Jolla, CA 92093-0112, USA ({\tt pgill@ucsd.edu}). }

\date{}

\makeatletter

\newcommand{\leqnomode}{\tagsleft@true}
\newcommand{\reqnomode}{\tagsleft@false}

\newcommand{\m}{\phantom-}

\newcommand{\pmat}[1]{\begin{pmatrix}#1\end{pmatrix}} 

\ifx   \innerprod\undefined%
   \def\innerprod(#1,#2){\langle#1,#2\rangle} 
\fi

\newcommand{\Newcommand}[2]%
   {\ifx#1\undefined \newcommand{#1}{#2} \else \renewcommand{#1}{#2} \fi}
\ifx\mod\undefined
  \newcommand{\mod}[1]{|#1|}
\else
  \renewcommand{\mod}[1]{|#1|}
\fi
\ifx\Range\undefined
  \newcommand{\Range}{\mathop{\operator@font{range}}}
\else
  \renewcommand{\Range}{\mathop{\operator@font{range}}}
\fi
\Newcommand{\Re} {\mathbb{R}}           
\Newcommand{\subject}{\mathop{\operator@font{subject\ to}}}  
\Newcommand{\epi}{\mathop{\operator@font{epi}}}
\Newcommand{\graph}{\mathop{\operator@font{graph}}}

\providecommand{\det}  {\mathop{\operator@font{det}}}
\providecommand{\diag} {\mathop{\operator@font{diag}}}
\providecommand{\dim}  {\mathop{\operator@font{dim}}}
\providecommand{\exp}  {\mathop{\operator@font{exp}}}

\providecommand{\rank} {\mathop{\operator@font{rank}}}
\providecommand{\vec}  {\mathop{\operator@font{vec}}}

\newsavebox\myboxA
\newsavebox\myboxB
\newlength\mylenA

\newcommand*\xoverline[2][0.75]{%
    \sbox{\myboxA}{$\m@th#2$}%
    \setbox\myboxB\null
    \ht\myboxB=\ht\myboxA%
    \dp\myboxB=\dp\myboxA%
    \wd\myboxB=#1\wd\myboxA
    \sbox\myboxB{$\m@th\overline{\copy\myboxB}$}
    \setlength\mylenA{\the\wd\myboxA}
    \addtolength\mylenA{-\the\wd\myboxB}%
    \ifdim\wd\myboxB<\wd\myboxA%
       \rlap{\hskip 0.5\mylenA\usebox\myboxB}{\usebox\myboxA}%
    \else
        \hskip -0.5\mylenA\rlap{\usebox\myboxA}{\hskip 0.5\mylenA\usebox\myboxB}%
    \fi}
\newcommand{\newstrut}{\rule[-1.25ex]{0pt}{4ex}}%
\newcommand{\newstrutl}{\rule[-1.25ex]{0pt}{3ex}}%
\newcommand{\newstrutu}{\rule{0pt}{3.25ex}}%

\newcommand{\abs}[1]{|#1|}
\newcommand{\projX}{\mathop{\operator@font{\mathbf{Proj}_\setX}}}
\newcommand{\ProjX}{\mathop{\operator@font{Proj_\setX}}}
\newcommand{\Proj}[1]{\mathop{\operator@font{Proj_{\mathit{#1}}}}}
\newcommand{\Prox}[1]{\mathop{\operator@font{Prox_{\mathit{#1}}}}}
\newcommand{\aff}{\mathop{\operator@font{aff}}}
\newcommand{\cl}{\mathop{\operator@font{cl}}}
\newcommand{\ri}{\mathop{\operator@font{ri}}}
\newcommand{\Angle}{\mathop{\operator@font{angle}}}
\newcommand{\argmin}{\mathop{\operator@font{argmin}}}
\newcommand{\argmax}{\mathop{\operator@font{argmax}}}
\newcommand{\conv}{\mathop{\operator@font{conv}}}
\newcommand{\convclosure}{\mathop{\operator@font{\overline{conv}}}}
\newcommand{\cone}{\mathop{\operator@font{conic}}}
\newcommand{\dom}{\mathop{\operator@font{dom}}}
\newcommand{\eig}{\mathop{\operator@font{eig}}}
\newcommand{\svSet}{\mathop{\operator@font{sv}}}

\newcommand{\interv}{\mathop{\operator@font{\mathbf{interv}}}}
\newcommand{\algmult}{\mathop{\operator@font{am}}}
\newcommand{\geomult}{\mathop{\operator@font{gm}}}
\newcommand{\dottimes}{\mathop{\mbox{\,\boldmath$\cdot$\,}}}
\renewcommand{\dottimes}{\,\mathop{\operator@font{\cdot}}\,}

\newcommand{\T}{^{\hthinsp\mathrm{\scriptstyle T}}}

\newcommand{\invT}{^{-\mathrm{\scriptstyle T}}}

\newcommand{\realpart}{\mathop{\operator@font{Re}}\,}
\newcommand{\imaginarypart}{\mathop{\operator@font{Im}}\,}
\newcommand{\sgap}{\;}
\newcommand{\mgap}{\;\;}
\newcommand{\bgap}{\;\;\;}

\newcommand{\minimize}[1]{{\displaystyle\minim_{#1}}}

\newcommand{\Bigp}[1]{\bigl(#1\bigr)}

\newcommand{\tmattt}[4]{\big(\, #1 \ \  #2 \ \  #3 \ \  #4\,\big)}

\newcommand{\mattt}[4]{\Bigp{\; #1 \quad #2 \quad #3 \quad #4\;}}

\newcommand{\Cpp}{C\raise3pt\hbox{\tiny++}}
\newcommand{\blackslug}{\hbox{\hskip 1pt \vrule width 4pt height 6pt depth 1.5pt
  \hskip 1pt}}
\newcommand{\boxit}[1]{\vbox{\hrule\hbox{\vrule\hskip 3pt
  \vbox{\vskip 3pt #1 \vskip 3pt}\hskip 3pt\vrule}\hrule}}

\newcommand{\maxim}{\mathop{\operator@font{maximize}}}
\newcommand{\minim}{\mathop{\operator@font{minimize}}}
\newcommand{\cond}{{\operator@font{cond}}}
\newcommand{\defined}{\mathop{\,{\scriptstyle\stackrel{\triangle}{=}}}\,}
\newcommand{\definedas}{\defined}
\newcommand{\closure}{\mathop{\operator@font{closure}}}
\newcommand{\sep}{\mathop{\operator@font{sep}}}
\newcommand{\dist}{\mathop{\operator@font{dist}}}
\newcommand{\In}{\mathop{\operator@font{In}}}
\newcommand{\Signature}{\mathop{\operator@font{Sn}}}
\newcommand{\boundary}{\mathop{\operator@font{bnd}}}
\newcommand{\interior}{\mathop{\operator@font{int}}}
\newcommand{\Int}{\mathop{\operator@font{int}}}
\newcommand{\Null}{\mathop{\operator@font{null}}}
\newcommand{\columnRank}{\mathop{\operator@font{column\ rank}}}
\newcommand{\rowRank}{\mathop{\operator@font{row\ rank}}}
\newcommand{\Rank}{\mathop{\operator@font{rank}}}
\newcommand{\op}{\mathop{\operator@font{op}}}
\newcommand{\re}{\mathop{\operator@font{re}}}
\newcommand{\range}{\mathop{\operator@font{range}}}
\newcommand{\sign}{\mathop{\operator@font{sign}}}
\newcommand{\sgn}{\mathop{\operator@font{sgn}}}
\newcommand{\Span}{\mathop{\operator@font{span}}}
\newcommand{\trace}{\mathop{\operator@font{trace}}}
\newcommand{\drop}{^{\null}}
\newcommand{\etal}{et al.}  

\newcommand{\Grad}{\nabla\!}
\newcommand{\Hess}{\nabla^2\!}

\newcommand{\cents}{\hbox{\operator@font{\rlap/c}}}
\newcommand{\half}  {{\textstyle\frac12}}

\newcommand{\inv}{^{-1}}

\newcommand{\norm}[1]{\|#1\|}

\newcommand{\Seq}[1]{\{\hthinsp#1\hthinsp\}}

\newcommand{\Set}[1]{\big\{\, #1 \,\big\}}
\newcommand{\spose}[1]{\hbox to 0pt{#1\hss}}

\newcommand{\words}[1]{\mgap\text{#1}\mgap}
\newcommand{\wordss}[1]{\quad\text{#1}\quad}

\newcommand{\hthinsp}{\mskip  1   mu}    
\newcommand{\pthinsp}{\mskip  2   mu}    

\newcommand{\dr}[1]{_{#1}\drop}

\newcommand{\supmax}{^{\operator@font{max}}}
\newcommand{\supmin}{^{\operator@font{min}}}
\newcommand{\submax}{_{\operator@font{max}}}
\newcommand{\submaxd}{_{\operator@font{max}}^{\null}}
\newcommand{\submin}{_{\operator@font{min}}}
\newcommand{\submind}{_{\operator@font{min}}^{\null}}

\newcommand{\subplus }{_{\scriptscriptstyle +}}

\newcommand{\A}{_{\scriptscriptstyle A}}

\newcommand{\C}{_{\scriptscriptstyle C}}

\newcommand{\kp}[1]{_{k+#1}}

\newcommand{\jp}[1]{_{j+#1}}
\newcommand{\jm}[1]{_{j-#1}}
\newcommand{\ip}[1]{_{i+#1}}
\newcommand{\im}[1]{_{i-#1}}
\newcommand{\kd}{_k\drop}

\newcommand{\id}{_i\drop}

\newcommand{\idm}[1]{_{i-#1}\drop}
\newcommand{\jd}{_j\drop}
\newcommand{\jdp}[1]{_{j+#1}\drop}
\newcommand{\jdm}[1]{_{j-#1}\drop}

\newcommand{\M}{_{\scriptscriptstyle M}}

\Newcommand{\R}{_{\scriptscriptstyle R}}

\newcommand{\starsymbol}{\ast}
\newcommand{\superstar}{^\starsymbol}

\newcommand{\SUPERSTAR}{^{\raise 0.5pt\hbox{$\starsymbol$}}}
\renewcommand{\SUPERSTAR}{\superstar}

\newcommand{\pstar}{p\superstar}

\newcommand{\deltahat}{\skew6\widehat\delta}

\newcommand{\gammabar}{\bar\gamma}

\newcommand{\omegabar}{\bar\omega}

\newcommand{\sigmabar}{\skew{2.8}\bar\sigma}

\providecommand{\varGamma}  {{\mathit\Gamma}}

\newcommand{\Gammait}{\varGamma}

\newcommand{\setX}{\mathcal{X}}

\newcommand{\subS}{_{\scriptscriptstyle S}}

\newcommand{\dbar}{\bar d}

\newcommand{\Htilde}{\widetilde H}

\newcommand{\rbar}{\skew3\bar r}

\newcommand{\rtilde}{\skew3\widetilde r}

\newcommand{\cf}{\textrm{cf.}}

\providecommand{\CUTEst}{{\small CUTE}st}

\providecommand{\GLTR  }{{\small GLTR}}

\providecommand{\GLTR}{{\small GLTR}}

\makeatother

\input algmacros

\newcommand{\obj}{_{\rm{obj}}}
\newcommand{\res}{_{\rm{res}}}

\let\strutu=\strut
\newlength{\argwidth}
\newcommand{\tbc}[2]{
\settowidth{\argwidth}{#1}
\makebox[\argwidth][c]{#2}}

\newcommand{\chibar}{\bar\chi}
\newcommand{\psibar}{\bar\psi}

\newcommand{\LDLT}{LDL$^{\mathrm{T}}$}

\newcommand{\objQP}{\varphi}
\newcommand{\objQPtilde}{\widetilde\varphi}

\newcommand{\eqq}{=\!=}
\newcommand{\cgstep}{\omega}
\newcommand{\cgstepbar}{\omegabar}
\newcommand{\subtol}{_{\mathit{tol}}}

\newcommand{\md}{_m\drop}

\newcommand{\mdm}[1]{_{m-#1}\drop}

\renewcommand{\Grad}{\nabla}

\begin{document}

\maketitle\thispagestyle{empty}

\begin{abstract}
\noindent
For a strongly convex unconstrained minimization problem, Newton's method
with a line search offers rapid convergence to a minimizer by incorporating
second-order information. However, for a general unconstrained problem the
Hessian need not be positive definite, which can lead to the inability to
find a search direction or to search directions that fail to produce
descent. Modified Newton methods are designed to handle this
situation. This paper considers modifications to a
Newton-conjugate-gradient approach based on Lanczos tridiagonalization.
Existing methods are extended using dynamic modification schemes that
provide a descent direction while keeping the modification bounded. A
Lanczos-based scheme and a conjugate-gradient (CG)-based scheme are
formulated and analyzed. These methods are mathematically equivalent in
exact arithmetic, but their practical behavior differs in finite
precision. Finally, it is shown how the convergence and computational cost
are affected by varying the amount of additional curvature introduced and
by allowing the early termination of the inner iterations.
\end{abstract}


\section{Introduction}
This paper concerns the formulation and analysis of  modified Newton methods
for the unconstrained optimization problem
\[
     \minimize{x\in \Re^n}\mgap f(x),
\]
where $f$ is a twice-continuously differentiable scalar-valued
function.
The focus is on line-search methods, which generate a sequence $\Seq{x_k}$
with $x\kp1 = x_k + \alpha_k p_k$, where $\alpha_k$ is a positive scalar
and $p_k$ is a direction of decrease for $f$ at $x_k$.

If the Hessian matrix $\Hess f(x_k)$ is sufficiently positive definite
(i.e., the least eigenvalue is sufficiently positive), the search direction
$p_k$ may be computed by solving the Newton equations $\Hess f(x_k) p_k = -
\Grad f(x_k)$, which constitute the first-order optimality conditions for
the minimizer of the quadratic model
\begin{equation}    \label{eqUC:quadratic-model}
   \objQP_k(p) = f(x_k) + p\T\Grad f(x_k) + \tfrac12 p\T \Hess f(x_k) p.
\end{equation}
The Newton equations may be solved directly using conventional matrix
factorization methods.  For example, an efficient method is based on the
\LDLT{} factorization $\Hess f(x\kd) = L\kd D\kd L_k\T$ in which $L_k$ is
unit lower triangular and $D_k$ is diagonal.  If the number of variables is
large and $\Hess f(x\kd)$ is sparse, then there are very efficient
sparse-matrix methods for computing the factors of a symmetrically permuted
version of $\Hess f(x\kd)$.  The efficiency of these methods follows from
the fact that there exist permutations such that $L_k$ and $D_k$ are also
sparse. However, factorization methods may be impractical if $n$ is large
and $\Hess f(x\kd)$ is not sparse.  Such matrices arise in methods for
large-scale constrained minimization that involve the solution of equations
with matrices of the form $Z\T\Hess f(x) Z$, where $\Hess f(x_k)$ and $Z$
may be sparse but the product is dense.  These cases are examples of
problems for which iterative methods for the solution of the relevant
equations are the only option.  In this context the matrix associated with
the linear equations is used only in the computation of matrix-vector
products.

For a nonconvex objective function, the Hessian may not be positive
definite. As a result, the conventional Newton direction may fail to be a
descent direction or, if the Hessian is singular, the Newton equations may
not have a solution. Modified Newton methods handle this problem by
changing the Newton equations to obtain a more suitable search
direction. The aim is to improve the behavior of the method when the
Hessian is not positive definite, while retaining the fast local
convergence of Newton's method when the Hessian is sufficiently positive
definite.

A common approach is to replace the Hessian by a related positive-definite
matrix. Early Newton-type methods based on modified Cholesky factorizations
were developed by Gill and Murray \cite{GilM74}. Kaniel and
Dax~\cite{KanD79} proposed a related approach based on a symmetric
decomposition of the Hessian, which is used to construct a descent
direction when the Hessian is not positive definite. This idea was later
developed further through alternative ways of modifying the
factorization. Examples include the methods of Schnabel and Eskow
\cite{SchE90,SchE99} and Cheng and Higham \cite{CheH98}. Fang and
O'Leary~\cite{FanO08} give a broad overview and comparison of modified
Cholesky methods.

Other methods retain and use information about negative curvature instead
of removing it completely. For example, Forsgren, Gill, and Murray
\cite{ForGM95} use a partial Cholesky factorization to compute both a
descent direction and a direction of negative curvature. Mor\'{e} and
Sorensen \cite{MorS79} use a symmetric indefinite factorization to obtain
directions of negative curvature. For large-scale problems, similar ideas
appear in Newton-CG methods, where conjugate gradients are used to find an
approximate solution of a damped or regularized Newton system while also
detecting nonpositive curvature. The method may then use either an
approximate Newton direction or a direction of negative curvature. Such
approaches can also provide guarantees for approaching points that satisfy
approximate second-order optimality conditions (see Royer, O'Neill and
Wright~\cite{RoyOW20} and Yao \etal~\cite{YaoXRWM23}).

Regularized and hybrid methods provide other ways of handling an indefinite
Hessian. These include adaptively shifting the Hessian and combining
regularized Newton directions with directions of negative curvature (see,
e.g., Shen, Chen and Liang~\cite{SheCL12}, Yang~\cite{Yan15}, Zhu and
Xiao~\cite{ZhuX24}, and Zhou \etal~\cite{ZhoXLBDZ25}). Although these
methods construct their search directions in different ways, they all aim
to produce reliable steps when the Hessian is indefinite and to recover
Newton-like behavior near a sufficiently regular local minimizer.

This work builds on the modified Newton method proposed by
Nash~\cite{Nas84}, which applies a two-stage modification to the Lanczos
tridiagonalization of the Hessian. The modification is constructed so that
the resulting matrix retains sufficient positive curvature while the
magnitude of the applied modifications remains bounded. As the method
operates on the tridiagonal matrix generated by the Lanczos process, it
requires only quantities that become available during the Lanczos
iterations.

In Nash's original method, a fixed positive parameter $\delta$ is used both
to determine whether a modification is required and to specify a lower
bound on the diagonal entries of the \LDLT{} factorization of the modified
tridiagonal matrix. Choosing a small value of $\delta$ avoids unnecessary
modifications, but may produce a modified matrix with very small positive
curvature in several directions. This observation motivates the separation
of the threshold used to determine whether the curvature is sufficiently
large from the target curvature imposed when a modification is required.

In this paper, Nash's two-stage modification of the tridiagonal matrix is
extended by introducing iteration-dependent target curvatures that are
distinct from the fixed curvature threshold.  Bounds on the resulting
modifications are established when the sequence of target curvatures is
uniformly bounded.  Two choices of the iteration-dependent target curvature
are proposed, both of which use information from the current tridiagonal
matrix.  In addition, a gradient-based CG formulation of the modification
is derived and shown to be equivalent to the Lanczos formulation in exact
arithmetic.  Finally, numerical results are presented to facilitate a
comparison of the proposed methods and demonstrate the effects of using
several criteria for the early termination of the inner iterations.

The remainder of the paper is organized as
follows. Section~\ref{secUC:notation} introduces the notation used
throughout the paper. Section~\ref{secUC:background} gives a brief review
of the gradient-based method of conjugate gradients (CG) and its
Lanczos-based formulation, and describes the modification proposed by
Nash. Section~\ref{secUC:lanczos} introduces the iteration-dependent target
curvatures and establishes a bound on the resulting
modifications. Section~\ref{secUC:cg} derives the equivalent gradient-based
CG formulation. Section~\ref{sec:numerical} presents the numerical results,
and Section~\ref{secUC:conclusion} provides some concluding remarks.

\section{Notation}\label{secUC:notation}
In order to simplify the notation, the discussion will focus on the generic
quadratic function
\[
  \objQP(p)  = g\T p + \half p\T H p,
\]
where $g$ and $H$ denote the gradient $\Grad f(x_k)$ and Hessian $\Hess
f(x_k)$, respectively (\cf~(\ref{eqUC:quadratic-model})).  The methods are
iterative, and given $p_0\in \Re^n$ generate a finite sequence of inner iterates
\[
  p\jp1 = p_j + \cgstep_j q_j,     \words{for} j = 0, 1, \dots, \words{with}
   q_j\in\Re^n \words{and} \cgstep_j\in\Re.
\]
The vector $r(p)$ is used to denote $g + H p$, the gradient of $\objQP$,
with $r_j = r(p_j)$.  Under the assumption that $H$ is positive definite,
if $r(p_j) = 0$ then $p_j$ is the minimizer $\pstar$ and the iterative
method terminates.  Unless otherwise stated, the scalar $\cgstep$ denotes
the step length in the inner (Lanczos) iterations.  For clarity, the index
used for $p$ is $j$ for the inner iterations and $k$ for the outer
iterations.

Unless explicitly indicated otherwise, $\norm{\cdot}$ denotes the vector
two-norm or its induced matrix norm.  The least eigenvalue of a symmetric
matrix $H$ will be denoted by $\lambda\submin(H)$.  The vectors $e$ and
$e_j$ denote, respectively, the column vector of ones and the $j$th column
of the identity matrix $I$; the dimensions of $e$, $e_j$ and $I$ are
defined by the context.

\section{Preliminaries}\label{secUC:background}
This section starts with a review of the class of conjugate-gradient
iterative methods for the solution of the positive-definite Newton
equations $\Hess f(x_k) p_k = - \Grad f(x_k)$.  This is followed by a
review of iterative methods based on the Lanczos process
(Lanczos~\cite{Lanc50}) and it is shown how these methods may be used to
formulate iterative methods in which an indefinite $\Hess f(x\kd)$ is
modified implicitly to be positive definite.  In all the methods to be
considered, the matrix associated with the linear equations is used only in
the computation of matrix-vector products.

\subsection{The method of conjugate gradients.}
All the methods considered are either explicitly or implicitly related
to the method of conjugate gradients proposed by Hestenes and
Stiefel~\cite{HesS52}.  The method minimizes the quadratic function
$\objQP(p)$ by generating a sequence of search directions $\Set{q_j}$ that
are conjugate with respect to $H$, i.e., $q_i\T H q\jd = 0$ for $i \ne j$.
Starting from an initial point $p_0$, the iterates are generated according
to
\[
   p_j = p\jm1 + \cgstep\jm1 q\jm1,
\]
where $\cgstep\jm1$ is the step length obtained from an exact line search
along $q\jm1$. The search directions are constructed iteratively from the
gradient vectors $r_0,r_1,\dots,r\jm1$ so that the resulting iterates
minimize $\objQP(p)$ over a sequence of expanding subspaces.  In exact
arithmetic, the method terminates at the minimizer after at most $m$
iterations, where $m \le n$. The method is summarized in
Algorithm~\ref{algUC:CG-classical-cg}. To distinguish this formulation from
the Lanczos-based method of conjugate gradients, the procedure in
Algorithm~\ref{algUC:CG-classical-cg} will be referred to as
\texttt{grad\_CG}.

\begin{algorithm}[ht]
\caption{\label{algUC:CG-classical-cg}
\hspace{-3pt}.\mgap The method of conjugate gradients.\algstrut}

\begin{algorithmic}[0]
       \Procedure{$[p] = \ $\algname{grad\_CG}}{$H$, $g$, $p_0$, $\tau\subtol$\Astrutu}
\State    \Restrictions{$\tau\subtol > 0$, \sgap $H$ positive definite};
\State    $r_0 = g + H p_0$;                     \agap
             $j = 0$;
          \While{$\norm{r_j} > \tau\subtol$}
             \If{\ent{j \eqq 0}}
\State          $q_j = -r_j$;
             \Else
\State          $q\jd = -r\jd + \big(r_j\T H q\jdm1/q\jm1\T H q\jdm1\big) q\jdm1$;
             \EndIf
\State       $\cgstep\jd = -r_j\T q\jd/q_j\T H q\jd$;    \agap
             $p\jdp1    = p\jd + \cgstep\jd q\jd$;       \agap
             $r\jdp1    = r\jd + \cgstep\jd H q\jd$;
\State       $j        \gets j + 1$;
          \EndWhile
\State    $p = p_j$;
       \EndProcedure
\end{algorithmic}
\end{algorithm}
The subspace spanned by the search directions $q_0$, $q_1$, \dots, $q\jdm1$
is precisely the Krylov subspace generated by $H$ and the initial gradient
$r_0$, namely
\[
  \Span\Set{q_0, q_1, \dots, q\jdm1 } = \Span\Set{r_0, H r_0, \dots, H^{j-1}r_0 }.
\]
It follows that the quadratic function is minimized over a sequence of expanding
Krylov subspaces.

\subsection{A method of conjugate gradients based on Lanczos tridiagonalization.}
An alternative implementation of the method of conjugate gradients may be
derived using Lanczos tridiagonalization (see Paige and
Saunders~\cite{PaiS75}).  Given $v_1 = -g/\beta_1$ with $\beta_1 =
\norm{g}$, the Lanczos vectors are computed from the sequence
\[
  \beta_2 v_2     = H v_1 - \gamma_1 v_1,  \qquad
  \beta\jp1 v\jp1 = H v_j - \gamma_j v_j - \beta_j v\jm1,
                               \quad j = 2, 3, \dots,
\]
where $\beta\jp1$ is chosen to give $\norm{v\jp1}=1$, and $\gamma\jd =
v_j\T H v\jd$.  The Lanczos process generates a sequence of tridiagonal
matrices $T_j$, $j = 1$, $2$, \dots, $m$, such that
\[ 
  H V\jd = V\jd T\jd + \beta\jdp1 v\jdp1 e_j\T,
\] 
with
\[ 
   V_j = \tmattt{ v_1}{ v_2}{ \cdots}{ v_j}          \wordss{and}
   T_j = \pmat{ \gamma_1 & \beta_2  &          &         &         &         \\
                \beta_2  & \gamma_2 & \beta_3  &         &         &         \\
                         & \beta_3  & \gamma_3 & \beta_4 &         &         \\
                         &          & \ddots   & \ddots  & \ddots  &         \\
                         &          &          & \ddots  & \ddots  & \beta_j \\
                         &          &          &         & \beta_j & \gamma_j  }.
\] 
In exact arithmetic, the Lanczos iteration will terminate with $\beta_{m+1}
= 0$ for some integer $m \le n$, which implies that $T\md = V_m\T H
V\md$ with $V_m\T V\md = I\md$.

A set of conjugate vectors may be computed using a suitable factorization
of $T_j$.  If $H$ is positive definite, then so is $T_j$, which has a
unique \LDLT{} factorization $T\jd = L\jd D\jd L_j\T$, with $D\jd$ diagonal
and $L\jd$ unit lower triangular.  Given $T_j$, the matrix $T\jp1$ has the
form
 \[
  T\jp1
   = \pmat{ L\jd                  &  0           \\[3pt]
            l\jdp1 e_j\T          &  1           }
     \pmat{ D\jd                  &  0           \\[3pt]
            0                     &  d\jdp1      }
     \pmat{ L_j\T                 &  l\jdp1 e\jd \\[3pt]
            0                     &  1           },
\]
where $l\jdp1 = \beta\jdp1/d\jd$ and $d\jdp1 = \gamma\jdp1 -
\beta\jp1^2/d\jd$.  If $j\ge 1$ then $l\jp1$ and $d\jp1$ are associated
with the \LDLT{} factors of a $2\times2$ matrix, i.e.,
\[
   \pmat{ d_j        & \beta\jp1 \\
          \beta\jp1  & \gamma\jp1  }
   =  \pmat{  1      & 0          \\
              l\jp1  & 1            }
      \pmat{ d_j     & 0          \\
               0     & d\jp1        }
      \pmat{ 1       & l\jp1      \\
             0       & 1            }.
\]
It remains to compute the conjugate vectors from the Lanczos vectors and
\LDLT{} factors.  With the definition $Q\jd = V\jd L_j\invT$, it holds that
\[ 
  Q_j\T H Q\jd
   =  L_j\inv V_j\T H V\jd  L_j\invT
   =  L_j\inv T\jd  L_j\invT
   =  D\jd,
\] 
and the columns of $Q\jd = V\jd L_j\invT$ are mutually conjugate with
respect to $H$.  An important property of $Q\jd$ is that its columns
may be computed using a simple two-term recurrence involving the Lanczos
vectors. If the columns of $Q_j$ are denoted by $q_0$, $q_1$, \dots,
$q\jm1$, then
\[
  \mattt{v\dr1}{v\dr2}{\cdots}{v\jd}
     = V\jd = Q\jd L_j\T
     = \mattt{q\dr0}{l\dr2 q\dr0 + q\dr1}{\cdots}{l\jd q\jdm2 + q\jdm1},
\]
which gives the two-term recurrence
\[ 
   q_0 = v_1, \qquad q_i = v\ip1 - l\ip1 q\im1, \mgap  i = 1, 2, \dots, j-1.
\] 
When an exact line search is used, the step length may be calculated
iteratively as
\begin{equation}                     \label{eqUC:def-cgstepj}
   \cgstep\jm1
       = \begin{cases}
               \m \phantom(\beta_1/d_1     & \hbox{if $j=1$;} \\
                - (\beta_j/d_j)\cgstep\jm2  & \hbox{otherwise.}
         \end{cases}
\end{equation}
With this method, the iterates $p\jd = p\jdm1 + \cgstep\jdm1 q\jm1$
minimize $\objQP$ over the same sequence of expanding subspaces as the
method of conjugate gradients, i.e., the iterates are identical. The
scalars $\beta_j$ and $d_j$ are positive, and the expression
(\ref{eqUC:def-cgstepj}) implies that the members of the sequence $\Set{
  \cgstep_i }$ alternate in sign and are \emph{not} the same as the step
lengths of the conjugate-gradient method.  A simple change to the Lanczos
process can be used to provide positive step lengths at every iteration.
In the computation of $\beta\jp1$ as $\norm{H v_j - \gamma_j v_j - \beta_j
  v\jm1}$ there is an implicit choice of sign for the calculation of the
norm.  If $\beta\jp1$ is computed as $-\norm{H v_j - \gamma_j v_j - \beta_j
  v\jm1}$ then $\cgstep\jm1 > 0$ for all $j > 0$.
A Lanczos-based method of conjugate gradients is summarized in
Algorithm~\ref{algUC:Lanczos-CG}.

\begin{algorithm}[ht] 
  \caption{\label{algUC:Lanczos-CG}
\hspace{-3pt}.\mgap Lanczos-based method of conjugate gradients.\algstrut}

\begin{algorithmic}[0]
      \Procedure{[$p] =\  $\algname{lanczos\_CG}}{$H$, $g$, $\tau\subtol$, $j\submax$\Astrutu}
\State     \Restrictions{$g \ne 0$, \sgap $H$ positive definite \sgap $\tau\subtol > 0$, \sgap $j\submax > 0$;}
\State     $p_0 = 0$;                         \agap
               $v_0 = 0$;
\State     $\beta\dr1 = \norm{g}$;            \agap
              $v\dr1 = -g/\beta\dr1$;         \agap
              $\gamma\dr1 = v_1\T H v\dr1$;   \agap
              $q\dr0 = v\dr1$;
\State     $d_1 = \gamma_1$;
\State     $j = 1$;
           \While{\ent{j < j\submax}}
\State        $s\jp1 = H v_j - \gamma_j v_j - \beta_j v\jm1$; \agap
                    $\beta\jp1 = - \norm{s\jp1}$;                    \label{algUC:Lanczos-CG-Step-7}
              \If{\ent{\mod{\beta\jp1} > \tau\subtol}}
\State           $v\jdp1 = s\jdp1/\beta\jdp1$;                \agap
                    $\gamma\jdp1 = v\jp1\T H v\jdp1$;
\State           $d\jp1 = \gamma\jdp1 - \beta\jp1^2/d\jd$;    \agap
                    $l\jp1 = \beta\jp1/d_j$;
\State           $q_j   = v\jp1 - l\jp1 q\jm1$;
              \EndIf
\State        \IIf{\ent{j = 1}} \ent{\cgstep_0 = \beta_1/d_1}  \IElse{\ent{\cgstep\jm1 = -  \cgstep\jm2\beta_j/d_j}};
\State        $p_j = p\jm1 + \cgstep\jm1 q\jm1$;   \Comment{The gradient norm at $p_j$ is $\mod{\beta\jp1}\,\cgstep\jm1$\,}
\State        \IF{} \ent{\mod{\beta\jp1} \le \tau\subtol} \ent{\BREAK};
\State        $j \gets j+1$;
           \EndWhile
\State     $p = p_j$;
       \EndProcedure
  \end{algorithmic}
\end{algorithm}

\subsection{Nash's modified Lanczos method}
An alternative to the conventional modified \LDLT{} factorization was
proposed by Nash~\cite{Nas84}. The method computes a diagonal modification
in a sequential two-stage procedure, using only the information from the
rows and columns of $T$ that have been computed thus far. Each diagonal
entry is modified twice -- the first modification ensures that the resulting
matrix is sufficiently positive definite, while the second controls the
magnitude of subsequent modifications by bounding the correction required
at the next iteration.

Initially, the value of $\gamma_1$, the leading diagonal of $T$, is
computed. If $H$ is not positive definite, this value may be negative or
zero.  To allow for this, the least nonnegative scalar $\sigmabar_1$ is
defined such that $\gammabar_1 = \gamma_1 + \sigmabar_1 \ge \delta$, i.e.,
\[
   \sigmabar_1 =  \max\Set{\gamma_1,\delta} - \gamma_1,
\]
where $\delta >0$ is chosen to be the smallest allowable positive value of
the modified diagonal entry.  (The value of $\delta$ determines the
minimum distance of the modified matrix to singularity.)

Once $v_2$ has been computed the new tridiagonal matrix may be modified if
necessary as
\[
    \pmat{ \dbar_1  & \beta_2    \\
           \beta_2  & \gamma_2   } +
    \pmat{ \sigma_1 &  0         \\
            0       & \sigmabar_2  }
 =  \pmat{  1       & 0          \\
           l_2      & 1            }
    \pmat{ d_1      & 0          \\
            0       & \dbar_2      }
    \pmat{  1       & l_2        \\
            0       & 1            },
\]
where $\sigma_1$ and $\sigmabar_2$ are nonnegative scalars computed so
that the quantity $\sigma_1 + \sigmabar_2$ is minimized. As $\dbar_1 =
\gamma_1 +\sigmabar_1$ (with $\dbar_1 \ge \delta$), this implies that
$\gamma_1$ may undergo two modifications during the factorization, which is
the reason for the use of bars to denote \emph{intermediate} values that
may be modified again at the next stage of the factorization.  The $j$th
stage ($j \ge 2$) involves the computation of the factorization
\[
    \pmat{ \dbar_j    & \beta\jp1   \\
           \beta\jp1  & \gamma\jp1   } +
    \pmat{ \sigma_j   &  0           \\
            0         & \sigmabar\jp1  }
 =  \pmat{ 1          & 0            \\
           l\jp1      & 1              }
    \pmat{ d_j        & 0            \\
              0       & \dbar\jp1      }
    \pmat{ 1          & l\jp1        \\
           0          & 1              },
\]
where $\sigma_j$ and $\sigmabar\jp1$ are nonnegative scalars defined so
that the quantity $\sigma_j + \sigmabar\jp1$ is minimized.  If the quantity
$\dbar\jp1$ is denoted by $\phi\jp1(\sigma_j, \sigmabar\jp1)$ to indicate
the dependence on $\sigma_j$ and $\sigmabar\jp1$, then
\[
  \phi\jdp1(\sigma\jd,\sigmabar\jdp1)
      \definedas  \dbar\jdp1
      = \gamma\jdp1 + \sigmabar\jdp1 - \beta\jp1^2/(\dbar\jd + \sigma\jd).
\]
Note that $\phi\jdp1(0,0) = \gamma\jdp1 - \beta\jp1^2/\dbar\jd$ is the
Schur complement computed by a conventional Cholesky factorization. If
$\phi\jp1(0,0)\ge \delta$, then $\phi\jdp1(0,0)$ is sufficiently positive,
$\sigma_j = \sigmabar\jp1 = 0$ and the factorization is not modified.
Otherwise, the factors are modified by finding $\sigma_j$ and
$\sigmabar\jp1$ that solve the subproblem
\begin{equation} \label{eqUC:Nash-subproblem}
\begin{aligned}
\minimize{\sigma_j,\sigmabar\jp1}
           &\bgap \sigma_j +  \sigmabar\jp1                                   \\
\subject   &\bgap \sigma_j \ge 0, \bgap \sigmabar\jp1  \ge 0, \bgap
                      \phi\jp1(\sigma_j,\sigmabar\jp1) \ge \delta.
\end{aligned}
\end{equation}

The search directions and iterates computed with the modified tridiagonal
matrix will then coincide with the iterates generated by
Algorithm~\ref{algUC:Lanczos-CG} using the modified Hessian
\[
  \Htilde\mdm1 = H + \sum_{i=1}^{m-1} (\sigmabar\id+\sigma\id) v\id v_i\T + \sigmabar_m\drop v_m\drop  v_m\T,
\]
where $m$ denotes the index such that $\beta_{m+1} = 0$.

\section{Extensions of Nash's Method}\label{secUC:lanczos}
Nash's two-stage modification ensures that the pivots of the modified
$LDL\T$ factorization are bounded below by a prescribed constant
$\delta>0$, while also bounding the magnitude of the diagonal
modification. However, if $\delta$ is fixed at a small value, the
resulting matrix may have very small positive curvature in the direction of
several Lanczos vectors. The solution of the modified Newton equations may
then have large components in these directions, i.e., directions in which
the original Hessian has small or negative curvature.  Moreover, the
resulting solution may have an excessively large norm. The line search in
the outer iterations may then require many function evaluations before an
acceptable step size is obtained.

Choosing a uniformly larger value of $\delta$ could reduce this effect, but
may impose unnecessary modifications. This justifies the separation of the threshold
used to detect insufficient curvature from the target curvature imposed
after modification. A fixed small value of $\delta$ is retained to
determine whether the current Schur complement is sufficiently positive. If
this test fails, the target curvature is replaced by an iteration-dependent
value $\delta\jp1 \ge \delta$. Apart from the choice of $\delta\jp1$,
the two-stage modification remains unchanged.  It is shown  that, provided
the values $\delta\jp1$ are uniformly bounded above, the resulting
diagonal modifications retain a corresponding uniform bound.

At the $j$th stage of Nash's method, a modified \LDLT{} factorization
of the $2\times 2$ matrix
\begin{equation} \label{eqUC:Bj-defined}
 B_j =  \pmat{ \dbar_j   & \beta\jp1  \\
               \beta\jp1 & \gamma\jp1   }
\end{equation}
is computed, and the unmodified Schur complement $\phi\jp1(0,0) = \gamma\jp1 -
\beta\jp1^2/\dbar\jd$ is compared to the fixed detection threshold
$\delta$. If $\phi\jp1(0,0) \ge \delta$, then the current curvature is
considered sufficient and no modification is made.  In this case, no target
curvature is required. Conversely, when $\phi\jp1(0,0) < \delta$, an
iteration-dependent target curvature $\delta\jp1 \ge \delta$ is selected
and used as the lower bound for the Schur complement in the modification
subproblem.

If a modification is necessary,  the subproblem analogous to
(\ref{eqUC:Nash-subproblem}) is solved, where $\delta$ is replaced by
$\delta\jdp1$, i.e.,
\begin{equation} \label{eqUC:Nash-subproblem-with-dj}
\begin{aligned}
\minimize{\sigma_j,\sigmabar\jp1}
           &\bgap \sigma_j +  \sigmabar\jp1                                   \\
\subject   &\bgap \sigma_j \ge 0, \bgap \sigmabar\jp1  \ge 0, \bgap
                      \phi\jp1(\sigma_j,\sigmabar\jp1) \ge \delta\jdp1
\end{aligned}
\end{equation}
where $\phi\jp1(\sigma\jd,\sigmabar\jdp1) = \gamma\jdp1 + \sigmabar\jdp1 -
\beta\jp1^2/(\dbar\jd + \sigma\jd)$. Therefore the fixed detection
threshold $\delta$ is used only to determine whether the current curvature
is sufficient, whereas the target curvature $\delta\jp1$ specifies the
curvature added after the need for a modification has already been
established.

Nash's approach begins by testing the minimizer with respect to the
nonlinear constraint for feasibility. If this point is infeasible, the
solution is selected from the two remaining candidates by choosing the one
that minimizes the objective, noting that at least one candidate is always
feasible. Here a slightly different approach is used, where first-order
information from the nonlinear constraint is utilized and a greedy
selection strategy using its partial derivatives is employed. An advantage
of this derivation is that the inequalities needed to establish the bounds
on the modifications are immediately apparent.  The method is summarized
below (for a thorough discussion of greedy selection methods see Bretthauer
and Shetty~\cite{BreS02}).  As the objective function is linear in both
variables with equal coefficients, the solution is dependent on the partial
derivatives
\[
   \nabla_{\sigma_j}\phi\jp1(\sigma_j, \sigmabar\jp1)
          = \frac{\beta\jp1^2}{(\dbar_j + \sigma_j)^2}, \wordss{and}
   \nabla_{\sigmabar\jp1}\phi\jp1(\sigma_j,\sigmabar\jp1)
          = 1.
\]
The partial derivatives are equal for the values
\[
 \sigma_j = \beta\jp1 - \dbar_j, \qquad \sigmabar\jp1 = \delta\jp1 - \gamma\jp1 + \beta\jp1,
\]
which is the required solution of (\ref{eqUC:Nash-subproblem-with-dj}) if
the nonnegativity constraints are satisfied.  If $\beta\jp1 - \dbar_j <
0$, then for any feasible $\sigma_j \ge 0$ it holds that
$\nabla_{\sigma_j}\phi\jp1(\sigma_j, \sigmabar\jp1) < 1$ and the solution is
given by
\[
  \sigma\jd = 0, \qquad \sigmabar\jdp1 = \delta\jdp1 - \gamma\jdp1 + \beta\jp1^2/\dbar\jd.
\]
Otherwise the nonnegativity constraint is violated by $\delta\jp1 -
\gamma\jp1 + \beta\jp1 < 0$, in which case for any feasible
$\sigmabar\jp1$, it holds that $\nabla_{\sigma_j}\phi\jp1(\sigma_j,
\sigmabar\jp1) > 1$ and the solution is given by
\[
  \sigma_j = \beta\jp1^2/(\gamma\jp1 - \delta\jp1) - \dbar_j, \qquad \sigmabar\jp1 = 0.
\]
The resulting procedure for finding modifications based on the bound
$\delta\jp1$ is described in Algorithm~\ref{algUC:New-Nash-modification}.

\begin{algorithm}[ht]
  \caption{\label{algUC:New-Nash-modification}
\hspace{-3pt}.\mgap  The tridiagonal modification.\algstrut}

\begin{algorithmic}[0]
       \Procedure{[$\dbar\jdp1, \sigma_j\superstar, \sigmabar\jp1\superstar] =\  $\algname{modT}}{$\dbar_j$, $\gamma\jp1$, $\beta\jp1$\Astrutu}
\State   Choose $\delta > 0$;
\State   $\phi = \gamma\jdp1 - \beta\jp1^2/\dbar\jd$;     \Comment{$\phi = \phi\jdp1(0,0)$}
         \If{\ent{\phi \ge \delta}}
\State      $\sigma_j\superstar = 0$;                      \agap
               $\sigmabar\superstar\jp1 = 0$;              \agap
               $\dbar\jdp1 = \phi$;
         \Else
\State      Choose the target curvature $\delta\jp1 \ge \delta$; \Comment{See Algorithm~\ref{algUC:Nash-minD}}
\State      $\dbar\jp1  = \delta\jp1$;
\State      Set $\sigma = \abs{\beta\jp1} - \dbar_j$;      \agap
                $\sigmabar = \delta\jp1 + \abs{\beta\jp1} - \gamma\jp1$;
            \If{\ent{\sigma \ge 0 \AND \sigmabar \ge 0}}
\State         $\sigma_j\superstar = \sigma$;              \agap
                  $\sigmabar\jp1\superstar = \sigmabar$;   \newstrut
            \ElsIf{\ent{\sigmabar < 0}}
\State          $\sigma_j\superstar = \dbar\jd\big(\delta\jp1 - \phi\big)/\big(\gamma\jdp1 - \delta\jp1\big)$; \agap
                   $\sigmabar\jp1\superstar = 0$;          \newstrut
            \Else
\State          $\sigma_j\superstar = 0$;                  \agap
                   $\sigmabar\jp1\superstar = \delta\jp1 - \phi$;
            \EndIf
         \EndIf
       \EndProcedure
  \end{algorithmic}
\end{algorithm}

A modified Lanczos method based on the two-stage modification with distinct
bounds on the curvature detection threshold $\delta$ and the
iteration-dependent target curvature $\delta\jp1$ is given in
Algorithm~\ref{algUC:Lanczos-modCG}. When a modification is necessary, the
value $\delta\jp1$ may be chosen freely, provided that it is bounded below
by $\delta$. To keep the modifications bounded, a safeguard $\delta\submax$
is imposed so that the target curvature cannot grow arbitrarily large. This
implies that the target curvature satisfies the bounds $\delta \le
\delta\jp1 \le \delta\submax$.

\begin{algorithm}[ht]
  \caption{\label{algUC:Lanczos-modCG}
\hspace{-3pt}.\mgap Lanczos-based modified conjugate-gradient method.\algstrut}

\begin{algorithmic}[0]
      \Procedure{[$p] =\  $\algname{lanczos\_mod\_CG}}{$H$, $g$, $\tau\subtol$, $j\submax$\Astrutu}
\State     \Restrictions{$g \ne 0$, \sgap $\tau\subtol > 0$, \sgap $j\submax > 0$;}
\State      $p_0 = 0$;                        \agap
               $v_0 = 0$;
\State      $\beta\dr1 = \norm{g}$;           \agap
              $v\dr1 = -g/\beta\dr1$;         \agap
              $\gamma\dr1 = v_1\T H v\dr1$;   \agap
              $q\dr0 = v\dr1$;
\State        Choose $\delta > 0$;
\State        $\sigmabar_1 =  \max\Set{\gamma_1,\delta} - \gamma_1$; \agap
              $\dbar_1 = \gamma_1 + \sigmabar_1$;
\State      $j = 1$;
           \While{\ent{j < j\submax}}
\State        $s\jp1 = H v_j - \gamma_j v_j - \beta_j v\jm1$; \agap
                    $\beta\jp1 = -\norm{s\jp1}$;
              \If{\ent{\mod{\beta\jp1} > \tau\subtol}}
\State           $v\jdp1 = s\jdp1/\beta\jdp1$;   \agap
                    $\gamma\jdp1 = v\jp1\T H v\jdp1$;
\State           $[\dbar\jp1,\sigma_j,\sigmabar\jp1] = \algname{modT}(\dbar_j,\gamma\jp1,\beta\jp1)$;
                     \Comment{The tridiagonal modification}
\State           $d_j  =  \dbar_j + \sigma_j$;  \agap
                    $l\jp1 = \beta\jp1/d_j$;
\State           $q_j   = v\jp1 - l\jp1 q\jm1$;
              \Else
\State           $d_j          =  \dbar_j$;
              \EndIf
\State        \IIf{\ent{j = 1}} \ent{\cgstep_0 = \beta_1/d_1}  \IElse{\ent{\cgstep\jm1 = -  \cgstep\jm2\beta_j/d_j}};
\State        $p_j = p\jm1 + \cgstep\jm1 q\jm1$;   \Comment{The gradient norm at $p_j$ is $\mod{\beta\jp1}\cgstep\jm1$\,}
\State        \IF{} \ent{\mod{\beta\jp1} \le \tau\subtol} \ent{\BREAK};
\State        $j \gets j+1$;
           \EndWhile
\State     $p = p_j$;
       \EndProcedure
  \end{algorithmic}
\end{algorithm}

The iterates generated by Algorithm~\ref{algUC:Lanczos-modCG} depend on the
entries of the modified tridiagonal matrix. At iteration $j$, the two-stage
modification changes the diagonal entries $\dbar_j$ and $\dbar\jp1$,
producing $d_j=\dbar_j+\sigma_j$ and a new value of $\dbar\jp1$. The search
directions $\Set{ q_0$, $q_1$, \dots, $q\jm1 }$ and the step lengths
$\cgstep_0$, $\cgstep_1$, \dots, $\cgstep_{j-2}$ have already been
determined and are not affected by this modification. The modified value
$d_j$ determines the current step length $\cgstep\jm1$, whereas $\dbar\jp1$
does not occur in the update
\[
   p_j = p\jm1 + \cgstep\jm1 q\jm1.
\]
Suppose that the Lanczos process terminates at iteration $m$, i.e.,
$\beta_{m+1} < \tau\subtol$.  On completion of iteration $j$ ($j <m$) the
modified Hessian is given by
\begin{equation} \label{eqUC:Lanczos-modified-Hessian}
 \Htilde\jd
     = H + \sum_{i=1}^j (\sigmabar\id + \sigma\id) v\id v_i\T
                   + \sigmabar\jdp1 v\jdp1 v\jp1\T.
\end{equation}
Then, in exact arithmetic, the iterates $\Set{ p_1$, $p_2$, \dots, $p_j}$
generated by Algorithm~\ref{algUC:Lanczos-modCG} coincide with the first
$j$ iterates generated by Algorithm~\ref{algUC:Lanczos-CG} applied to
$\Htilde\jd$. Thus, through iteration $j$, the modified Lanczos-based
method may be interpreted as applying the conventional Lanczos-based method
of conjugate gradients to the current modified Hessian $\Htilde\jd$.

The modifications generated by Nash's original method are bounded
(Nash~\cite[Theorem 2]{Nas84}). The next step is to extend this result to
iteration-dependent curvature targets. In the implementation, the target
curvatures are safeguarded by the prescribed upper bound $\delta\submax$,
so that
\[
   0 < \delta \le \delta\jp1 \le \delta\submax
\]
at every stage at which a modification is required. The lower bound ensures
that the modified pivots remain sufficiently positive, while the upper
bound makes it possible to control the modifications uniformly over the
iterations. No additional assumption is made on how the values $\delta\jp1$
are selected. Under these conditions, the modifications generated by
Algorithm~\ref{algUC:New-Nash-modification} satisfy the following bound.

\begin{proposition}
Let $T \in \Re^{m \times m}$ denote the symmetric tridiagonal matrix
generated by the Lanczos process applied to a symmetric matrix $H$. Assume
that $T+ E = L D L\T$ is a modified Cholesky factorization with
modification computed using
{\rm Algorithm~\ref{algUC:New-Nash-modification}}.
In addition, assume that $\delta \le \delta\jp1 \le \delta\submax$ for all
$0 \le j < m$. Then
\[
\begin{aligned}
        \sigmabar\jp1 &\le \delta\jp1 + \abs{\gamma\jp1} + \mod{\beta\jp1}, &\qquad j &= 0, 1, \dots, m-1, \\
        \sigma_j      &\le \mod{\beta\jdp1},                                &\qquad j &= 1, 2, \dots, m-1,
\end{aligned}
\]
so that
\[
     \norm{E} \le \delta\submax + \gamma + 2\beta,   \wordss{where}
    \gamma= \max_{1 \le j \le m}\Set{\mod{\gamma_j} } \wordss{and} \beta = \max_{1 \le j \le m}\Set{\mod{\beta_j}}.
\]
\end{proposition}
\begin{proof}
For the initial modification, if $\gamma_1 \ge \delta$ then $\sigmabar_1=
0$. Otherwise $\sigmabar_1 = \delta_1 - \gamma_1$, giving the bound
$\sigmabar_1 \le \delta_1 + \mod{\gamma_1}$.

For $j \ge 1$, if $\gamma\jdp1 - \beta\jp1^2/\dbar\jd \ge \delta$, then
$\sigma_j = \sigmabar\jp1 = 0$, and the bounds on $\sigma_j$ and
$\sigmabar\jp1$ hold trivially.  Otherwise the solution
$(\sigma_j,\sigmabar\jdp1)$ is one of the following
\begin{subequations}
\begin{align}
     \sigma_j &= \mod{\beta\jdp1} - \dbar_j, &\qquad \sigmabar\jdp1 &= \delta\jdp1 - \gamma\jdp1 + \mod{\beta\jdp1}
            \label{eqUC:Ebounda}        \\
     \sigma_j &= 0,                          &\qquad \sigmabar\jdp1 &= \delta\jdp1 - \gamma\jdp1 + \beta\jp1^2/\dbar_j     \label{eqUC:Eboundb}        \\
     \sigma_j &= \beta\jp1^2/(\gamma\jdp1 - \delta\jdp1) - \dbar_j, &\qquad \sigmabar\jdp1 &= 0.
            \label{eqUC:Eboundc}
\end{align}
\end{subequations}
If the $(\sigma_j,\sigmabar\jdp1)$ of (\ref{eqUC:Ebounda}) is feasible, then
the nonnegativity of $\dbar_j$ implies that
\[
   \sigma_j      \le \mod{\beta\jdp1}                        \wordss{and}
   \sigmabar\jp1 \le \delta\jp1 + \mod{\gamma\jp1} + \mod{\beta\jdp1}.
\]
Alternatively, if $(\sigma_j,\sigmabar\jdp1)$ of (\ref{eqUC:Ebounda}) is
infeasible, then either $\mod{\beta\jdp1} \le \dbar_j$ or $\mod{\beta\jdp1} \le
\gamma\jp1 - \delta\jp1$.  If $\mod{\beta\jdp1} \le \dbar_j$, then the
solution is given by (\ref{eqUC:Eboundb}), with
\[
 \sigmabar\jdp1   = \delta\jdp1 - \gamma\jdp1 + \beta\jp1^2/\dbar\jd
                \le \delta\jdp1 - \gamma\jdp1 + \mod{\beta\jdp1}.
\]
As $\gamma\jdp1 - \beta\jp1^2/\dbar\jd < \delta < \delta\jp1$ and
$\sigmabar\jdp1$ is nonnegative, the required bound on $\sigmabar\jdp1$
follows directly.

If $\mod{\beta\jdp1} \le \gamma\jp1 - \delta\jp1$, then the solution is given
by (\ref{eqUC:Eboundc}), with
\[
     \sigma\jd = \beta\jp1^2/(\gamma\jdp1 - \delta\jdp1) - \dbar\jd \le \mod{\beta\jdp1} - \dbar\jd,
\]
in which case $\sigma\jd\le \mod{\beta\jdp1}$. It follows that $\norm{E}
\le \delta\submax + \gamma + 2\beta$, where $\gamma= \max_{1 \le j \le m}
\Set{\abs{\gamma_j}}$ and $\beta = \max_{1 \le j \le m}\Set{ \mod{\beta_j} }$.
\end{proof}

\subsection{Iteration-dependent target curvature}
The fixed value $\delta$ is used to determine whether the current Schur
complement is sufficiently positive. If $\gamma\jp1 - \beta\jp1^2/\dbar\jd
\ge \delta$ then no modification is necessary. When the detection test
fails, an iteration-dependent target curvature is computed. In order to
allow the imposed curvature to reflect the local properties of the matrix
being modified, the target curvatures are based on information available
from the current matrix being modified. This allows the target curvature to
vary throughout the factorization.

Two choices of $\delta\jdp1$ are proposed. The first candidate, denoted by
\textbf{Schur}, is the magnitude of the unmodified Schur complement
$\mod{\gamma\jdp1 - \beta\jp1^2/\dbar\jd}$. The second, denoted by
\textbf{Eigen}, is the minimum absolute eigenvalue of $B_j$
(\cf~(\ref{eqUC:Bj-defined})). In both cases, the choice of $\delta\jdp1$ is
safeguarded below by $\delta$ and above by $\delta\submax$ before it is
used in Algorithm~\ref{algUC:New-Nash-modification}. The corresponding
definitions are summarized in Algorithm~\ref{algUC:Nash-minD}.

\begin{algorithm}[H] 
  \caption{\label{algUC:Nash-minD}
\hspace{-3pt}.\mgap Two alternative choices of $\delta\jp1$.\algstrut}

\begin{algorithmic}[0]
       \Procedure{[$\delta\jp1] =\  $\algname{posCurve}}{$\dbar_j$, $\gamma\jp1$, $\beta\jp1$, $\delta$, $\delta\submax$\Astrutu}
                                                        \Comment{\textbf{Schur}}
\State   $\deltahat\jp1 = \mod{\gamma\jdp1 - \beta\jp1^2/\dbar\jd}$; \agap $\delta\jp1 = \min\Set{\max{\Set{\deltahat\jdp1,\delta}},\delta\submax}$;
       \EndProcedure

\State
       \Procedure{[$\delta\jp1] =\  $\algname{posCurve}}{$\dbar_j$, $\gamma\jp1$, $\beta\jp1$, $\delta$, $\delta\submax$\Astrutu}
                                                        \Comment{\bf{Eigen}}
\State   $\tau      = (\gamma\jp1 - \dbar_j)/(2\beta\jp1)$;
\State   $t         = \sign(\tau)/\big(\mod{\tau} + \sqrt{1 + \tau^2}\big)$;\algstrut
\State   $\lambda_1 = \dbar_j    - t\beta\jp1$;         \agap\algstrut
            $\lambda_2 = \gamma\jp1 + t \beta\jp1$;
\State   $\deltahat\jp1 = \min\Set{\mod{\lambda_1},\mod{\lambda_2}}$; \agap $\delta\jp1 = \min\Set{\max{\Set{\deltahat\jdp1,\delta}},\delta\submax}$;
       \EndProcedure
  \end{algorithmic}
\end{algorithm}

When the test in Algorithm~\ref{algUC:New-Nash-modification} detects
insufficient curvature, Algorithm~\ref{algUC:Nash-minD} can be used to
determine the curvature target $\delta\jdp1$ used in the modification
subproblem.

The developments of this section have been presented as extensions of
Nash's modified Lanczos method because the curvature information required
by the two-stage modification arises naturally from the tridiagonal matrix
generated by the Lanczos process.  Nevertheless, the Lanczos and
conjugate-gradient methods are closely related and generate the same
sequence of Krylov subspaces in exact arithmetic. This observation suggests
that the proposed modifications need not be restricted to a Lanczos-based
implementation. In the next section, it is shown that the quantities required
by the two-stage modification may be recovered directly from the
gradient-based method of conjugate gradients. The resulting formulation
provides a gradient-based implementation that is equivalent to the modified
Lanczos method of Algorithm~\ref{algUC:Lanczos-modCG}.

\section{A Modified Method of Conjugate Gradients}\label{secUC:cg}
The modifications introduced in Section~\ref{secUC:lanczos} are formulated
using the tridiagonal matrices generated by the Lanczos process as an
extension of the original method proposed by Nash. Nevertheless, the
Lanczos and gradient-based methods of conjugate gradients are closely
related and, in exact arithmetic, generate the same sequence of expanding
Krylov subspaces. This suggests that the quantities required by the
two-stage modification may also be recovered directly from the
gradient-based CG.

A direct gradient-based implementation is not immediate. In the Lanczos
formulation, the next Lanczos vector may be generated independently of
finding the next iterate. Consequently, the curvature information required
to complete the two-stage modification is available before the step length
and the next iterate are computed. In the gradient-based CG formulation,
however, the next gradient needed to complete the modification
is available only after a step has been taken. Moreover, the second
modification in the direction of the current gradient changes the curvature
along the current search direction and hence changes the step length used
to generate the next gradient. This gives an apparent circular dependence:
the completed modification is required to determine the step length,
whereas the next gradient required to compute that modification is
available only after the step has been taken.

In this section it is shown that this dependence may be resolved by first
computing a provisional gradient. A rank-one modification in the direction
of the current gradient changes the scale, but not the direction, of the
gradient obtained following an exact line search. This ensures that the
provisional gradient provides the direction of the next gradient and may be
used to recover the curvature information needed to complete the
modification.  It is then shown that the resulting gradient-based
formulation generates the same modified Hessians as the Lanczos formulation
in exact arithmetic. Finally, an implementation is derived that requires
only one additional vector and one additional scalar relative to the
conventional gradient-based method of conjugate gradients.

Throughout this section, $p_j$, $r_j$, and $q_j$ denote the current
iterate, gradient, and search direction, respectively, with the initial
iterate $p_0 = 0$.

\subsection{Recovery of curvature information}
In exact arithmetic, the gradients generated by CG are mutually
orthogonal. The normalized gradients therefore define the orthonormal
matrix
\[
   R\jp1 = \pmat{ \dfrac{r_0}{\|r_0\|} & \dfrac{r_1}{\|r_1\|} & \cdots & \dfrac{r_j}{\|r_j\|}},
\]
such that the matrix $R\jp1\T H R\jdp1$ is tridiagonal. This follows from
the mutual orthogonality of the gradients, which successively form a basis
for the expanding Krylov subspaces. This implies that the diagonal and
off-diagonal entries of $R\jp1\T H R\jdp1$ may be recovered directly from
the CG gradients. In particular,
\[
 \gamma\jdp1 = \frac{r_j\T H r_j}{\norm{r_j}^2}, \wordss{and}
 \beta\jdp1  = \frac{r_j\T H r\jdm1}{\norm{r_j}\norm{r\jdm1}} \words{for} j\ge 1.
\]

It follows that the diagonal modifications generated by the two-stage
procedure may be expressed in terms of the normalized CG gradients. After
the first modification associated with $r_j$ has been introduced, the
provisional modified Hessian has the form
\[
  \Htilde\jd = H + \sum_{i=1}^j \frac{\sigmabar\id+\sigma\id} {r_{i-1}\T r\idm1} r\idm1 r_{i-1}\T + \frac{\sigmabar\jp1}{r_j\T r\jd} r\jd r_j\T.
\]

The next stage of the modification will add a rank-one update in the
direction of $r\jd$ with curvature $\sigma\jp1$, and hence changes the
curvature along $q_j$ and the step length taken in that direction. In order
to determine $\sigma\jp1$, the two-stage procedure also requires curvature
information associated with the next normalized gradient
direction. However, this direction would ordinarily be available only
\emph{after} the completed step has been taken. The following result shows
that it is sufficient to compute a provisional gradient using $\Htilde\jd$.

\begin{proposition}\label{prop:parallel-gradients}
Assume that $p \in \Re^n$ such that $r = g + H p$ is orthogonal to $p$, and
consider search direction $q$ such that $q\T r = -r\T r$. Suppose that
curvature is added to $H$ via the rank-one update
\[
 \Htilde = H + \frac{\sigma}{r\T r}r r\T,
\]
where both $q\T H q$ and $q\T \Htilde q$ are nonzero. Then the gradient
$r\subplus$ found using an exact line search with $H$ is parallel to the
gradient $\rtilde\subplus$ found using $\Htilde$. In particular,
\[
 \rtilde\subplus = \frac{q\T H q}{q\T H q + \sigma r\T r}r\subplus.
\]
\end{proposition}
\begin{proof}
Using $\Htilde$ to calculate the step size in the exact line search yields the gradient
\[
   \rtilde\subplus = r + \left(\frac{r\T r}{q\T Hq + \sigma r\T r}\right)(Hq - \sigma r)
\]
whereas using $H$ yields
\[
   r\subplus = r + \frac{r\T r}{q\T Hq}Hq.
\]
Therefore
\begin{align*}
   \frac{q\T H q}{q\T H q + \sigma r\T r}r\subplus & = \frac{q\T H q}{q\T H q + \sigma r\T r}\left(r + \frac{r\T r}{q\T Hq}Hq\right) \\
   & = \left(\frac{q\T H q + \sigma r\T r - \sigma r\T r}{q\T H q + \sigma r\T r}\right)r + \left(\frac{r\T r}{q\T H q + \sigma r\T r}\right)Hq \\
   & = r + \left(\frac{r\T r}{q\T Hq + \sigma r\T r}\right)(Hq - \sigma r) = \rtilde\subplus.
\end{align*}
It follows that $\rtilde\subplus = \dfrac{q\T H q}{q\T H q + \sigma r\T
  r}r\subplus$, as required.
\end{proof}

\noindent
Proposition~\ref{prop:parallel-gradients} resolves the apparent circular
dependence. An exact line search generates a new gradient orthogonal to the
previous search space, which implies that the action of $\Htilde\jdp1$ on
the iterate $p\jd$ is identical to that of
\[
  \Htilde\jd + \frac{\sigma\jdp1} {r_j\T r\jd} r\jd r_j\T.
\]
Therefore a provisional step $\cgstepbar\jd$ may be computed first using
$\Htilde\jd$, yielding
\[
   \rbar\jdp1 = r\jd + \cgstepbar\jd\Htilde\jd q\jd.
\]
The role of $\cgstepbar\jd$ is to generate a vector parallel to the
gradient that would be obtained using the subsequent modified Hessian
$\widetilde H\jp1$. Consequently, the normalized provisional gradient may
be used to compute the quantities
\[
  \gamma\jp2 = \frac{\rbar\jp1\T H\rbar\jdp1}{\norm{\rbar\jdp1}^2}, \qquad
   \beta\jp2 = \frac{\rbar\jp1\T H r_j} {\norm{\rbar\jdp1}\norm{r_j}}.
\]
These quantities determine $\sigma\jp1$ and $\sigmabar\jp2$ as in
Algorithm~\ref{algUC:New-Nash-modification}.

\subsection{Calculating the modified gradient information}
A direct implementation of the gradient-based formulation would require all
previous gradients to be stored because the gradient, search direction and
step length are all calculated using the action of the modified Hessian on
the current search direction. However, the action of the accumulated
modifications on the current search direction may be represented instead
using one additional vector and one additional scalar.  To show this, let
$q_0$, $q_1$, \dots, $q_j$ denote the search directions generated by the CG
algorithm. It holds that
\[
 q_j\T r\id = -r_j\T r\jd \wordss{for} i \le j.
\]
It is straightforward to establish that for any rank-one modification in the
direction of $r\id$ for $i \le j$, the action on $q\jd$ may be recovered as
\[
 \frac{\sigma}{\norm{r\id}^2}r\id r_i\T q\jd
    = -\norm{r\jd}^2\frac{\sigma}{\norm{r\id}^2}r_i.
\]
It follows that the information from the modifications may be stored in a vector
$\psi_j$ and scalar $\chi_j$ such that $\psi_0 = \mathbf{0}, \chi_0 = 0$
and  may be updated iteratively as
\begin{alignat*}{2}
    \psibar\jp1 &= \psi\jd + \sigmabar\jp1 r\jd/\norm{r\jd}^2, & \qquad \psi\jp1 &= \psibar\jp1 + \sigma\jp1 r\jd/\norm{r\jd}^2, \\
    \chibar\jp1 &= \chi\jd + \sigmabar\jp1/\norm{r\jd}^2,      & \qquad \chi\jp1 &= \chibar\jp1 + \sigma\jp1/\norm{r\jd}^2.
\end{alignat*}
This implies that the search directions and step lengths generated for the
modified Hessians $\Htilde\jd$ and $\Htilde\jdp1$ may be calculated using
the expressions
\begin{alignat*}{2}
  q_j\T \Htilde\jd q\jd &= q_j\T H q\jd + (r_j\T r\jd)^2\chibar\jdp1, &\qquad q_j\T \Htilde\jdp1 q\jd &= q_j\T H q\jd + (r_j\T r\jd)^2\chi\jdp1,\\[3pt]
\noalign{\leftline{while the gradients may be generated iteratively using the Hessian-vector product}\vskip 3pt}
        \Htilde\jd q\jd &= H q\jd - \norm{r\jd}^2\psibar\jdp1,        &\qquad       \Htilde\jdp1 q\jd &= H q\jd - \norm{r\jd}^2\psi\jdp1.
\end{alignat*}
It follows that the modified Hessian need not be formed explicitly.
Relative to the conventional gradient-based method of conjugate gradients,
generating the necessary information from modification requires one
additional vector, $\psi_j$, and one additional scalar, $\chi_j$. These
identities give a gradient-based implementation of the two-stage
modification. At each iteration, the first-stage modification is
incorporated into the provisional quantities $\psibar\jdp1$ and
$\chibar\jdp1$. A provisional gradient is then generated and used to
compute the additional curvature information required by
Algorithm~\ref{algUC:New-Nash-modification}. Once the second-stage
modification has been determined the step length, gradient, and search
direction are updated using the completed modified Hessian.  The resulting
procedure is given in Algorithm~\ref{algUC:Gradient-modCG}.

\begin{algorithm}[ht] 
  \caption{\label{algUC:Gradient-modCG}
\hspace{-3pt}.\mgap A gradient-based modified conjugate-gradient method.\algstrut}

\begin{algorithmic}[0]
      \Procedure{[$p] =\  $\algname{grad\_mod\_CG}}{$H$, $g$, $\tau\subtol$, $j\submax$\Astrutu}
\State     \Restrictions{$g \ne 0$, \sgap $\tau\subtol > 0$, \sgap $j\submax > 0$;}
\State      $p_0 = 0$;                        \agap
            $r_0 = g$;                        \agap
            $\chi\dr0 = 0$;                   \agap
            $\psi\dr0 = \mathbf{0}$;          \agap
\State        $\gamma\dr1 = r_0\T H r\dr0/\norm{r\dr0}^2$;   \agap
              $q\dr0 = -r\dr0$;
\State        Choose $\delta > 0$;
\State        $\sigmabar_1 =  \max\Set{\gamma_1,\delta} - \gamma_1$; \agap
              $\dbar_1 = \gamma_1 + \sigmabar_1$;
\State      $j = 1$;
            \While{\ent{j \le j\submax}}
\State        $\chibar\jd = \chi\jdm1 + \sigmabar\jd/\norm{r\jdm1}^2$; \agap
              $\psibar\jd = \psi\jdm1 + \sigmabar\jd r\jdm1/\norm{r\jdm1}^2$;
\State        $\cgstepbar\jdm1 = -r\jm1\T q\jdm1/\big(q\jm1\T Hq\jdm1 + \chibar\jd \norm{r\jdm1}^4\big)$;
\State    $\rbar\jd = r\jdm1 + \cgstepbar\jdm1 (Hq\jdm1 - \norm{r\jdm1}^2\psibar\jd\big)$
              \If{\ent{\norm{\rbar\jd} > \tau\subtol}}
\State           $\gamma\jdp1 = \rbar_j\T H \rbar\jd/\norm{\rbar\jd}^2$;     \agap
                    $\beta\jdp1 = \rbar_j\T H r\jdm1/\big(\norm{\rbar\jd}\norm{r\jdm1}\big)$;
\State           $[\dbar\jp1,\sigma\jd,\sigmabar\jp1] = \algname{modT}\big(\dbar\jd,\gamma\jp1,\beta\jp1\big)$;
                     \Comment{The tridiagonal modification}
\State           $\chi\jd = \chibar\jd + \sigma\jd/\norm{r\jdm1}^2$;                 \agap
                    $\psi\jd = \psibar\jd + \sigma\jd r\jdm1/\norm{r\jdm1}^2$;
\State           $\cgstep\jdm1 = -r\jm1\T q\jdm1/\big(q\jm1\T Hq\jdm1 + \chi\jd \norm{r\jdm1}^4\big)$;
\State           $r\jd = r\jdm1 + \cgstep\jdm1 \big(Hq\jdm1 - \norm{r\jdm1}^2\psi\jd\big)$
\State           $q\jd = -r\jd + \left(q\jm1\T Hr\jd/\big(q\jm1\T Hq\jdm1 + \chi\jd\norm{r\jdm1}^4\big)\right)q\jdm1$;
              \Else
\State           $\cgstep\jdm1 =  \cgstepbar\jdm1$; \agap $r\jd          =  \rbar\jd$;
              \EndIf
\State        $p\jd = p\jdm1 + \cgstep\jdm1 q\jdm1$;
\State        \IF{} \ent{\norm{r\jd} \le \tau\subtol} \ent{\BREAK};
\State        $j \gets j+1$;
           \EndWhile
\State     $p = p_j$;
       \EndProcedure
  \end{algorithmic}
\end{algorithm}

\subsection{Equivalence of the modified Lanczos and gradient-based formulations}
It has been shown that Nash's two-stage modification may be integrated into
a gradient-based CG framework. It remains to show that the 
quantities recovered from the provisional gradients produce the same
modifications as those generated by the Lanczos-based method.

The iterates generated by Algorithm~\ref{algUC:Gradient-modCG} depend on
the accumulated modification of the Hessian. At iteration $j$, the
two-stage procedure introduces additional curvature only in the current
gradient direction $r\jdm1$ and in the direction of the provisional
gradient $\rbar\jd$. Consequently, the search directions $q_0$, $q_1$,
\dots, $q_{j-2}$, the step lengths $\cgstep_0$, $\cgstep_1$, \dots,
$\cgstep_{j-2}$, and the iterates $p_0$, $p_1$, \dots, $p\jdm1$, remain
unchanged. The completed modification in the direction $r\jdm1$ determines
the current step length $\cgstep\jdm1$, and hence the next iterate
\[
  p_j = p\jdm1+\cgstep\jdm1 q\jdm1.
\]
If the algorithm terminates at iteration $m$, i.e., $\norm{r\md} <
\tau\subtol$, then the modified Hessian on completion of iteration $j$
($j<m$) is
\begin{equation} \label{eqUC:ModCG-modified-Hessian}
  \Htilde\jd
   = H + \sum_{i=1}^j \frac{\sigmabar\id+\sigma\id} {r_{i-1}\T r\idm1} r\idm1 r_{i-1}\T
                                + \frac{\sigmabar\jp1}{r_j\T r\jd} r\jd r_j\T.
\end{equation}
It follows that the iterates $p_0$, $p_1$, \dots, $p_j$, generated by
Algorithm~\ref{algUC:Gradient-modCG} coincide with the iterates generated
through iteration $j$ by Algorithm~\ref{algUC:CG-classical-cg} applied to
$\Htilde\jd$.  Therefore, equivalence of the Lanczos and gradient-based
modifications may be established by showing the equivalence of the modified
Hessian, which is done in the following proposition.

\begin{proposition}
Suppose that {\rm Algorithms~\ref{algUC:Lanczos-modCG}}
and~{\rm\ref{algUC:Gradient-modCG}} are applied in exact arithmetic to the
same symmetric matrix $H$ and vector $g$, and assume that the algorithms
use the same rule for selecting the target curvatures. Then, at every
completed inner iteration, the two algorithms generate the same
modifications and hence the same modified Hessian. Consequently, the
corresponding search directions are parallel and the iterates are
identical.
\end{proposition}
\begin{proof}
The proof is by induction on the inner iteration. At the initial iteration,
the first Lanczos vector satisfies
\[
   v_1 = -\frac1{\norm{g}} g =-\frac1{\norm{r_0}} r_0.
\]
Thus, $v_1$ and the first normalized CG gradient are parallel. The
initial curvature is the same in both formulations, i.e.,
\[
  \gamma_1 = v_1\T Hv\dr1 = \frac{r_0\T H r\dr0}{\norm{r_0}^2}.
\]
The value of $\sigmabar_1$ and the corresponding rank-one modification are
therefore identical.  Assume that the two formulations have generated
identical modifications through the completed modified Hessian at iteration
$j$, and that the normalized CG gradients are parallel to the corresponding
Lanczos vectors. The provisional gradient $\rbar\jdp1$ is computed using
$\Htilde\jd$. By equivalence of the Lanczos and gradient-based method of
conjugate gradients, $\rbar\jdp1$ is parallel to the Lanczos vector
$v\jdp2$ generated using $H$. It follows that the quantities
computed from the provisional gradient agree with the corresponding
Lanczos quantities:
\[
  \gamma\jp2 = \frac{\rbar\jp1\T H\rbar\jdp1}{\norm{\rbar\jdp1}^2}
             = v\jp2\T H v\jdp2, \qquad
   \beta\jp2 = \frac{\rbar\jp1\T H r_j} {\norm{\rbar\jdp1}\norm{r_j}} = v\jp2\T H v\jdp1.
\]
Hence the two formulations apply
Algorithm~\ref{algUC:New-Nash-modification} to the same tridiagonal data
and select the same target curvature. They therefore compute the same
values of $\sigma\jdp1$ and $\sigmabar\jp2$, implying that the modifications
introduced by the two algorithms are identical. This proves that the
completed modified Hessians agree at iteration $j+1$. The result follows by
induction. As the two algorithms generate identical modified Hessians,
their search directions are parallel and their exact line searches generate
identical iterates.
\end{proof}

The developments in this section show that Nash's two-stage modification
may be integrated into a gradient-based CG framework with the use of
provisional gradients to compute curvature information and the storage of
one extra vector and scalar to represent the accumulated modifications. In
exact arithmetic, the resulting modification generates an identical
modified Hessian as the Lanczos-based formulation in
Algorithm~\ref{algUC:Lanczos-modCG}, as well as parallel search directions
and identical iterates. In finite precision, the quantities are recovered
using different recurrence relations. The behavior of the two formulations
in finite precision is compared in the following section.

\section{Numerical Results}\label{sec:numerical}
The results were obtained for optimization problems from the \CUTEst{} test
collection (see Bongartz \etal{} \cite{BonCGT95} and Gould, Orban and
Toint~\cite{GouOT15}).  As of August 1, 2026, there are $291$ unconstrained
problems in the test set. A subset of $286$ problems was selected based on
the number of variables. In particular, a problem was chosen if the number
of variables was of the order of 5000 or less. The same criterion was used
to set the dimension of those problems for which the problem size can be
specified.  The 5 problems {\tt BA-L16LS}, {\tt BA-L21LS}, {\tt BA-L49LS},
{\tt BA-L52LS}, and {\tt BA-L73LS} have more than $23$K variables and were
excluded from the test set. A total of 200 of
the 286 problems exhibited an indefinite Hessian.

The runs were performed using \textsc{Matlab} version
R2024a on a System76 Prime Elite  Linux machine running Ubuntu 22.04 with a
Ryzen 9 9950X (16 Cores, 32 Threads) and 48 GB DDR5 RAM (5600 MHz).
The results were obtained using a \textsc{Matlab} implementation of
Algorithm~\ref{algUC:modified-Newton}, which implements a conventional
line-search modified Newton method with a first-order termination condition.  The method uses a
backtracking line search based on the affine model $\Gammait_k(x) = f(x_k) +
\Grad f(x_k)\T (x - x_k)$.

\begin{algorithm}[htb]
\caption{\label{algUC:modified-Newton}
\hspace{-3pt}.\mgap Modified Newton  method.\algstrut}

\begin{algorithmic}[0]
\State  \Constants{$j\submax$, $k\submax$; $0 < \eta\A < \half$;  $0<\gamma\C < 1$;\Astrutu}
\State  $k= 0$;
\State  Choose $x_0$; \agap   Evaluate $f(x_0)$ and $\Grad f(x_0)$;
        \While{$k \le k\submax \AND  \NOT{} \converged$}
\State    Use an iterative method to compute a search direction $p_k$; 
\State    $\alpha_k = 1$; \agap $\ell_k = 0$;
\State    $\rho_k   = \big(f(x_k) - f(x_k+\alpha_k p_k)\big)/
                            \big(\Gammait_k(x_k) - \Gammait_k(x_k+\alpha_k p_k)\big)$;
          \While{$\rho_k < \eta\A$}
\State      $\alpha_k \gets \alpha_k\gamma\C$; \agap $\ell_k \gets \ell_k + 1$;
\State      $\rho_k = \big(f(x_k) - f(x_k+\alpha_k p_k)\big)/
                            \big(\Gammait_k(x_k) - \Gammait_k(x_k+\alpha_k p_k)\big)$;
          \EndWhile
\State    $x\kp1 = x_k + \alpha_k p_k$; \agap
             Evaluate $\Grad f(x\kp1)$;
\State    $k\gets k+1$;
        \EndWhile
\end{algorithmic}
\end{algorithm}

For the proposed methods to be computationally viable for large-scale
problems, it is important to limit the number of inner iterations. Carrying
the Lanczos process to completion at every outer iteration may require a
large number of Hessian-vector products and may therefore incur a
substantial computational cost. Nevertheless, the curvature information
accumulated during the initial Lanczos iterations may be sufficient to
produce a more effective search direction than the negative gradient. Early
termination may thus reduce the cost of computing the search direction
while retaining some of the potential advantages of a Krylov-subspace
method. In the following, we compare three criteria for terminating the
Lanczos process before completion.

Nash~\cite{Nas84} and Nash and Sofer~\cite{NasS90} propose a criterion
based on the reduction in the quadratic objective function:
\begin{equation} \label{eqUC:early-term-Obj}
  j (\objQPtilde_j - \objQPtilde\jm1)/\objQPtilde_j \le \eta\obj.
\end{equation}
(In Appendix~\ref{secUC:objective} it is shown how the objective function
$\objQPtilde$ can be computed without explicit knowledge of the modified
Hessian).  Dembo, Eisenstat, and Steihaug~\cite{DemES82} propose
alternative criteria based on the relative residual norm
\begin{equation} \label{eqUC:early-term-Res}
   \norm{r_j} \le \eta\res \norm{g},
\end{equation}
and  the dynamic reduction of the relative norm of the
residual
\begin{equation} \label{eqUC:early-term-DynRes}
       \norm{r_j} \le \min\Set{ \half,\norm{g} } \norm{g}.
\end{equation}
This last criterion becomes progressively more stringent as $\norm{g}$
decreases.  The termination conditions (\ref{eqUC:early-term-Obj}),
(\ref{eqUC:early-term-Res}) and (\ref{eqUC:early-term-DynRes}) provide
alternative mechanisms for controlling the number of Lanczos iterations
and, consequently, the number of Hessian-vector products used to compute
the search direction.

In order to satisfy the conditions required for global convergence of the
modified Newton method, two additional early-termination conditions are
imposed.  The first is based on the magnitude of $\norm{p_j}$, and the
second is based on the value of the normalized directional derivative $-g\T
p_j/(\norm{g}\norm{p_j})$.  These conditions utilize two fundamental
properties of the conventional CG method, namely, that the norm of $p_j$ is
monotonically increasing (Steihaug~\cite{Ste83}), and that the normalized
directional derivative is monotonically decreasing (Dai, Mart{\'\i}nez and
Yuan~\cite{DaiMY03}). These results apply to the \emph{modified} Lanczos
and gradient-based formulations because the associated search directions
are identical to those generated by the Lanczos and CG methods applied to a
positive-definite modified Hessian
(\cf~(\ref{eqUC:Lanczos-modified-Hessian}) and
(\ref{eqUC:ModCG-modified-Hessian})). The specific termination criteria are
\[
   \norm{p_j} \ge p\submax   \wordss{and}
   -\frac{g\T p_j}{\norm{g}\norm{p_j}} \le c\submin.
\]

Figure~\ref{runs-Nash-extensions-6000} gives the performance profiles for
the Lanczos-based and CG-based methods based on Nash's original method and
the two proposed extensions. The profiles also include results from the
Lanczos trust-region method \GLTR{} (Gould \etal~\cite{GouLRT99}).  The
left figure gives the profiles for the number of function evaluations. The
right figure gives the profiles for the number of Hessian-vector
products. The iteration limit is $n+1$ in all cases.

\begin{figure}[ht]
\begin{center}
\includegraphics[width=6.5cm,height=6.5cm]{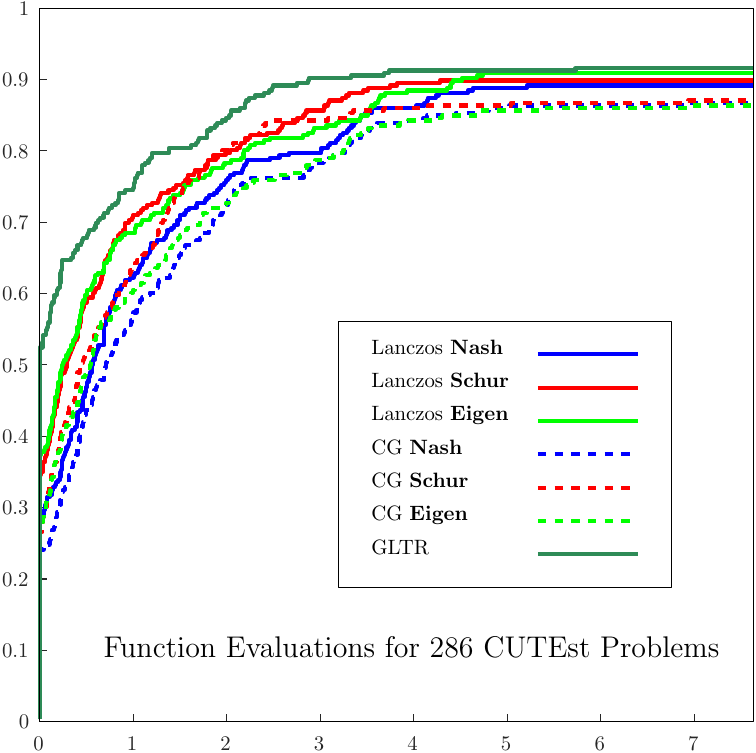}\mgap
\includegraphics[width=6.5cm,height=6.5cm]{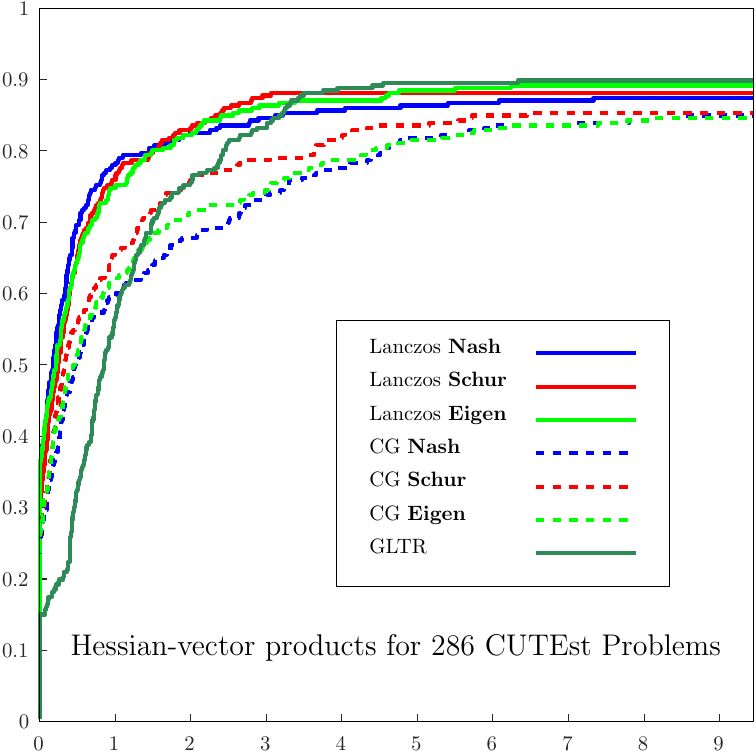}
\caption{\label{runs-Nash-extensions-6000}
\footnotesize Performance profiles for the Lanczos-based and CG-based
methods based on Nash's original method and the two proposed extensions.
The profiles also include results from the Lanczos trust-region method
GLTR.  The iteration limit is $n+1$ in all cases. The methods are applied
to 286 unconstrained problems from the CUTEst test set. The left figure
gives the profiles for the number of function evaluations. The right figure
gives the profiles for the number of Hessian-vector products.}
\end{center}
\end{figure}

\begin{figure}[ht]
\begin{center}
\includegraphics[width=6.5cm,height=6.5cm]{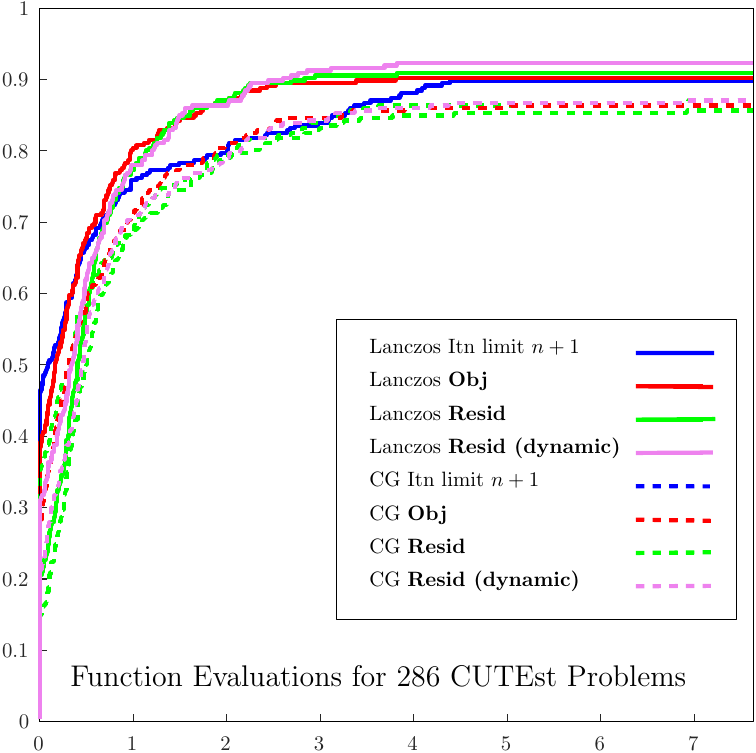}\mgap
\includegraphics[width=6.5cm,height=6.5cm]{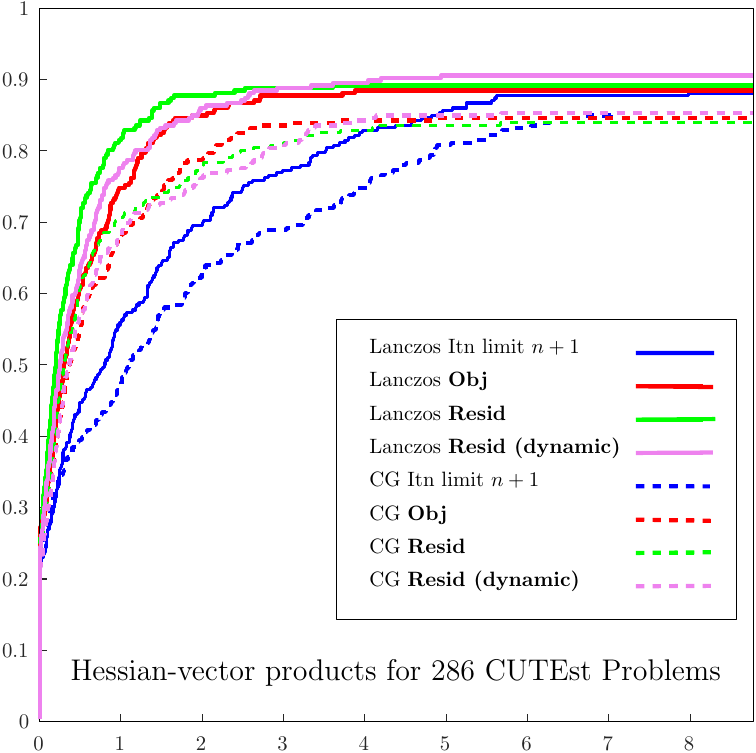}
\caption{\label{runs-All-Early}
\footnotesize Performance profiles for the Lanczos-based and CG-based
methods using the Schur-complement strategy with three early-termination
strategies (\ref{eqUC:early-term-Obj})--(\ref{eqUC:early-term-DynRes}),
labeled \texttt{Obj}, \texttt{Resid} and \texttt{Resid (dynamic)},
respectively. The parameters and tolerances for the methods with early
termination are given in Table~\ref{tabUC:params}.  The methods are applied
to 286 unconstrained problems from the CUTEst test set. The left figure
gives the profiles for the number of function evaluations. The right figure
gives the profiles for the number of Hessian-vector products.}
\end{center}
\end{figure}

A run was considered to be a success based on two criteria.  The first was
\[
  \norm{\Grad f(x_k)} \le \varepsilon\subS\big(1 + \mod{f(x_k)}\big)
  \wordss{or}
  \norm{\Grad f(x_k)} \le \epsilon\M^{1/2},
\]
where $\varepsilon\subS$ is a preassigned stationarity tolerance and
$\epsilon\M$ is the relative machine precision. If the problem is badly
scaled, a line search failure may occur before these conditions can be
satisfied.  In this case the run was considered a success if
\[
  \norm{\Grad f(x\kd)} \le  \epsilon\M^{1/2} \norm{\Grad f(x\dr0)}
  \wordss{or}
  \mod{f(x\kd)} \le \epsilon\M^{2/3}\big(1 + \mod{f(x\dr0)}\big).
\]
The specific values of the parameters used in the runs are given in
Table~\ref{tabUC:params}.
\begin{table}[htb]
\begin{center}
\begin{footnotesize}
    \begin{tabular}{|l|c|l|c|}\hline
    \multicolumn{1}{|l|}{\textbf{\strutu Parameter}}&
    \multicolumn{1}{c}{\tbc{Parameter}{\textbf{Value}}}&
    \multicolumn{1}{|l|}{\textbf{Parameter}}&
    \multicolumn{1}{c|}{\tbc{Parameter}{\textbf{Value}}}\newstrut \\ \hline \newstrutu
      Positive curvature threshold, $\delta$     & \texttt{ 1.0e-10 } &                                                           &                   \\
      Stationarity tolerance, $\varepsilon\subS$ & \texttt{ 1.0e-05 } & Lanczos tolerance, $\tau\subtol$                          & \texttt{ 1.0e-07 }\\
      Lanczos iteration limit, $j\submax$        & \texttt{ 6000 }    & Newton iteration limit, $k\submax$                        & \texttt{ 3000 }   \\
      Armijo step reduction factor,  $\gamma\C$  & \texttt{ 0.5 }     & Armijo step tolerance, $\eta\A$                           & \texttt{ 1.0e-04 }\\
      Early termination factor, $\eta\obj$       & \texttt{ 0.01 }    & Early termination factor, $\eta\res$                      & \texttt{ 0.5 }    \\
    Lower bound on $-g\T p_j/(\norm{g}\norm{p_j})$, $c\submin$& \texttt{ 0.01 } & Upper bound on $\norm{p_j}$, $p\submax$    & \texttt{ 1.0e+20 }\newstrutl   \\\hline
   \end{tabular}
\parbox[t]{.925\textwidth}{\caption{\label{tabUC:params}\small Constants and tolerances used across all problems and solvers.}}
\end{footnotesize}
\end{center}
\end{table}

First, the Lanczos implementation exhibits a small but consistent advantage
over the conjugate-gradient implementation. Although the performance
profiles are similar, the Lanczos-Nash method solves slightly more problems
and achieves slightly better overall efficiency. Second, the trust-region
method \GLTR{} appears to be the most robust method in terms of the total
number of problems solved. However, this increased robustness is generally
accompanied by a larger number of Hessian-vector products, indicating a
higher cost per solve.

The results also suggest that the choice of target curvature has a
noticeable effect on performance. Both adaptive strategies outperform the
fixed-curvature approach proposed by Nash over a substantial portion of the
test set. Among the adaptive variants, the Schur-complement strategy
(\textbf{Schur}) provides the strongest overall performance profile, whereas the
eigenvalue-based strategy (\textbf{Eigen}) solves a few additional problems and
appears to be slightly more robust. Overall, the results indicate that
exploiting iteration-dependent curvature information can improve efficiency
without sacrificing robustness.

Finally, the performance profiles in Figure~\ref{runs-All-Early} show that
early termination of the inner iteration is beneficial for all the
modification strategies considered. In each case, the use of an
early-termination criterion improves the balance between function
evaluations and Hessian-vector products when compared with carrying the
inner iteration to completion. This suggests that the curvature information
obtained during the initial Lanczos iterations is often sufficient to
compute an effective search direction, and that additional inner iterations
may provide only limited practical benefit.

\section{Conclusions}\label{secUC:conclusion}
A modified Newton method for large-scale unconstrained optimization has
been proposed in which two distinct curvature parameters are used. The
first determines when modification of the Hessian is required, whereas the
second specifies the target curvature imposed by the
modification. Separating these roles provides greater flexibility than
using a single parameter for both detection and modification. At the same
time, the proposed method retains a bound on the difference between the
true Hessian and the modified Hessian, thereby controlling the magnitude of
the modification.

Two iteration-dependent choices of the target curvature have been
proposed. These choices allow the imposed curvature to vary over the course
of the optimization process, rather than remaining fixed at all
iterations. We have also derived a conjugate-gradient-type implementation
of the modification that is theoretically equivalent to its Lanczos
formulation. The conjugate-gradient formulation avoids the explicit
construction of the Lanczos basis and provides an alternative
implementation using Hessian-vector products.

The different modification strategies have been evaluated numerically on a
collection of unconstrained optimization problems. The two proposed
iteration-dependent choices of the target curvature are compared with the
fixed target-curvature choice proposed by Nash. These comparisons are used
to examine the computational performance and robustness of the resulting
methods, and to investigate the effect of allowing the target curvature to
vary between iterations.

Three strategies for the early termination of the Lanczos iterations and
conjugate-gradient iterations have been considered.  Carrying the inner
iteration to completion may require a large number of Hessian-vector
products on large-scale problems, whereas the curvature information
obtained during the initial iterations may be sufficient to produce an
effective search direction.  Numerical experiments indicate that the use of
early-termination criteria can limit the cost of the inner iteration while
retaining the potential benefits of the modified Krylov-subspace direction.

\bibliographystyle{plain}
\bibliography{references}

\appendix

\section{Computing the Objective of the Modified Quadratic}\label{secUC:objective}
If the Lanczos-Nash algorithm terminates at iteration $m$, the final
modified Hessian $\Htilde_{m-1}$ may be written in the form
\[
  \Htilde\mdm1 = H + \sum_{i=1}^{m-1}(\sigma\id
                + \sigmabar\id) v\id v_i\T
                + \sigmabar_m\drop v_m\drop v_m\T,
\]
where $\Set{\sigma\id}_{i=1}^{m-1}$ and $\Set{\sigmabar\id}_{i=1}^m$ are
the diagonal modifications generated by
Algorithm~\ref{algUC:New-Nash-modification}.

An apparent difficulty of working with $\objQPtilde(p)$ is that the
quadratic term involves the final approximate Hessian $\Htilde_m$. The
following result shows that only the current modified Hessian is needed.

\begin{result}\label{resUC:H-equiv}
If $j < m$, then the matrix $\Htilde_{m-1}$ acts on $p_j$ as
\[
  H + \sum_{i=1}^j(\sigma\id + \sigmabar\id)v\id v_i\T,
\]
i.e., both matrices have the same action on  $p_j$.
\end{result}
\begin{proof}
As $p_0=0$, it follows that $p_j \in \Span\Set{v_1, v_2, \dots, v_j }$. The
orthogonality of the Lanczos vectors implies that
\begin{align*}
 \Htilde\mdm1 p\jd &= Hp\jd + \sum_{i=1}^{m-1}(\sigma\id + \sigmabar\id) v\id v_i\T p\jd
                         + \sigmabar_m\drop v_m\drop v_m\T p\jd \\
                 &= Hp\jd + \sum_{i=1}^j (\sigma\id + \sigmabar\id)v\id v_i\T p\jd
                  = \Big(H + \sum_{i=1}^j (\sigma\id + \sigmabar\id)v\id v_i\T \Big) p\jd,
\end{align*}
as required.
\end{proof}

\begin{result}\label{resUC:A2}
Given the iterates $p_0$, $p_1$, \dots, $p_j$ associated with the
Lanczos-Nash algorithm, it holds that
\[ 
  \objQPtilde(p\jdp1) =  - \tfrac12 p\jp1\T \Htilde\mdm1 p\jdp1
                      =    \objQPtilde(p\jd) - \tfrac12\cgstep_j\norm{\rtilde\jd},
\] 
where $p\jp1 = p_j + \cgstep_j q_j$ and $\rtilde_j = g + \Htilde\mdm1 p_j$.
\end{result}
\begin{proof}
If $\rtilde_j = \Grad \objQPtilde(p_j) = g + \Htilde_{m-1} p_j$, then
$\rtilde_j = \Big(H + \sum_{i=1}^j (\sigma\id + \sigmabar\id)v\id v_i\T \Big) p\jd$, as the iterates
of the Lanczos-Nash algorithm minimize $\objQPtilde(p)$ over a sequence of
expanding subspaces, it must hold that $p\jp1$ is orthogonal to $\Grad
\objQPtilde(p\jp1)$, i.e.,
\[
  0 = p\jp1\T(g + \Htilde\mdm1 p\jdp1) = \objQPtilde(p\jdp1) + \tfrac12 p\jp1\T \Htilde\mdm1 p\jdp1,
\]
which implies that $\objQPtilde(p\jdp1) = - \tfrac12 p\jp1\T \Htilde\mdm1
p\jdp1$.  The scalar $p\jp1\T \Htilde\mdm1 p\jdp1$ may be written in the form
\begin{align*}
  p\jp1\T \Htilde\mdm1 p\jdp1
    &= (p_j + \cgstep_j q_j)\T \Htilde\mdm1 (p_j + \cgstep_j q_j)\\
    &=  p_j\T \Htilde\mdm1 p\jd + 2\cgstep_j q_j\T \Htilde\mdm1 p\jd +  \cgstep_j^2 q_j\T \Htilde\mdm1 q\jd.
\end{align*}
As $p_j \in \Span\Set{ q_0, q_1, \dots, q\jm1 }$, the vector $q_j$ is
conjugate to $p_j$ and $p_j\T \Htilde\mdm1 q\jd = 0$.  It follows that
\[
  p\jp1\T \Htilde\mdm1 p\jdp1
     =  p_j\T \Htilde\mdm1 p\jd + \cgstep_j^2 q_j\T \Htilde\mdm1 q\jd.
\]

Given that $\cgstep\jd \Htilde_{m-1} q_j = \rtilde\jp1 - \rtilde\jd$, the
exact line search gives $q_j\T \rtilde\jdp1 = 0$ and it follows that $q_j\T
\rtilde\jd = -\norm{\rtilde\jd}$. Therefore we have that
\[
 \objQPtilde(p\jp1)
    = \objQPtilde(p_j) - \tfrac12 \cgstep_j\norm{\rtilde_j},
\]
which is what we wanted to show.
\end{proof}

Result~\ref{resUC:A2} shows that  updating the objective function value requires the norm of the
gradient $\rtilde\jd$. By Result~\ref{resUC:H-equiv} the norm of the
gradient $\rtilde\jd$ is equivalent to the gradient norm at $p\jd$ using
the modified Hessian at iteration $j$, i.e., $\norm{\rtilde\jd} =
\abs{\beta\jdp1}\cgstep\jdm1$.

\end{document}